\documentclass[sigconf]{acmart}

\usepackage{booktabs} % For formal tables

\setcopyright{rightsretained}
\usepackage{amsmath, amsfonts, amssymb, graphicx, color, soul, tikz, amsthm}
\usepackage{graphicx}
\usepackage{mathtools}
\usepackage[ruled]{algorithm2e}
\usepackage{enumerate}
\usepackage[shortlabels]{enumitem}
\usepackage{empheq}
\usepackage{multicol}

\usepackage{subcaption}
\usepackage[font=small]{caption}

\newcommand{\Eset}{\mathbb{E}}

\newcommand{\Rset}{\mathbb{R}}

\newcommand{\Ecal}{\mathcal{E}}
\newcommand{\Fcal}{\mathcal{F}}
\newcommand{\Gcal}{\mathcal{G}}

\newcommand{\Ncal}{\mathcal{N}}

\newcommand{\Pcal}{\mathcal{P}}
\newcommand{\Qcal}{\mathcal{Q}}

\newcommand{\Vcal}{\mathcal{V}}

\newcommand{\Xcal}{\mathcal{X}}

\newcommand{\Abf}{{\bf A}}

\newcommand{\Ebf}{{\bf E}}

\newcommand{\Gbf}{{\bf G}}

\newcommand{\Ibf}{{\bf I}}

\newcommand{\Qbf}{{\bf Q}}

\newcommand{\Vbf}{{\bf V}}
\newcommand{\Wbf}{{\bf W}}
\newcommand{\Xbf}{{\bf X}}
\newcommand{\Ybf}{{\bf Y}}

\newcommand{\abf}{{\bf a}}

\newcommand{\ebf}{{\bf e}}

\newcommand{\gbf}{{\bf g}}
\newcommand{\hbf}{{\bf h}}

\newcommand{\qbf}{{\bf q}}
\newcommand{\rbf}{{\bf r}}

\newcommand{\ubf}{{\bf u}}
\newcommand{\vbf}{{\bf v}}

\newcommand{\xbf}{{\bf x}}
\newcommand{\ybf}{{\bf y}}
\newcommand{\zbf}{{\bf z}}

\newcommand{\xibarbf}{{\bar{\xibf}}}

\newcommand{\ebarbf}{{\bar{\ebf}}}
\newcommand{\qbarbf}{{\bar{\qbf}}}
\newcommand{\xbarbf}{{\bar{\xbf}}}
\newcommand{\vbarbf}{{\bar{\vbf}}}

\newcommand{\1}{{\mathbf{1}}}

\newcommand{\xibf}{{\boldsymbol{\xi}}}
\newcommand{\Xibf}{{\boldsymbol{\Xi}}}

\newcommand{\xbar}{{\bar{x}}}

\newtheorem{prop}{Proposition}
\newtheorem{lem}{Lemma}
\newtheorem{thm}{Theorem}

\newtheorem{assump}{Assumption}
\newtheorem{remark}{Remark}

\newcommand{\redline}{\raisebox{2pt}{\tikz{\draw[-,red,solid,line width = 1.5pt](0,0) -- (6mm,0);}}}
\newcommand{\blueline}{\raisebox{2pt}{\tikz{\draw[-,blue,solid,line width = 1.5pt](0,0) -- (6mm,0);}}}

\newcommand{\blackline}{\raisebox{2pt}{\tikz{\draw[-,black,solid,line width = 1.5pt](0,0) -- (6mm,0);}}}

\begin{document}

\title{\LARGE \bf Distributed Stochastic Approximation  for Solving Network Optimization Problems Under Random Quantization
}

\author{{Thinh T. Doan$^{\dagger,\star}$, Siva Theja Maguluri$^{\star}$, Justin Romberg$^{\dagger}$}\\
$^\dagger$ School of Electrical and Computer Engineering\\
$^\star$ H. Milton Stewart School of Industrial and Systems Engineering\\
Georgia Institute of Technology, GA, 30332, USA\\ 
{\tt\small \{thinhdoan,\,siva.theja\}@gatech.edu, jrom@ece.gatech.edu.}}

%\date{}

\begin{abstract}
We study distributed optimization problems over a network when the communication between the nodes is constrained, and so information that is exchanged between the nodes must be quantized.  This imperfect communication poses a fundamental challenge, and this imperfect communication, if not properly accounted for, prevents the convergence of these algorithms.  Our first contribution in this paper is to propose a modified consensus-based gradient method for solving such problems using random (dithered) quantization.  This algorithm can be interpreted as a distributed variant of a well-known two-time-scale stochastic algorithm.  We then study the convergence and derive upper bounds on the rates of convergence of the proposed method as a function of the bandwidths available between the nodes and the underlying network topology, for both convex and strongly convex objective functions. Our results complement for existing literature where such convergence and explicit formulas of the convergence rates are missing. Finally, we provide numerical simulations to compare the convergence properties of the distributed gradient methods with and without quantization for solving the well-known regression problems over networks, for both quadratic and absolute loss functions.     
\end{abstract}

\maketitle

%!TEX root = random_quantization.tex

\section{Introduction}\label{sec:Intro}
In this paper, we consider optimization problems that are defined over a network of nodes\footnote{In this paper, nodes can be used to present for processors, robotics, or sensors.}. The objective function is composed of a sum of local functions where each function is known by only one node. In addition, each node is only allowed to interact with its neighboring nodes that are connected to it through the network. We assume no central coordination between the nodes and since each node knows only its local function, they are required to cooperatively solve the problems. This necessitates the development of distributed algorithms, which can be done under communication and computation constraints.

We are motivated by various applications of such problems within engineering. A standard example is the problem of estimating the radio frequency in a wireless network of sensors where the goal is to cooperatively estimate the radio-frequency power spectrum density through solving a regression problem  \cite{Giannakis2010}. In this application, the objective function is the total loss over the entire measured data by the sensors, which are scattered across a large geographical area.  Due to some privacy concerns, the sensors may not be willing to share their measurements, but only their own estimates, making distributed algorithms become necessary.

Another possible application is the problem of distributed information processing in \textit{edge (fog) computing}, which has recently received a surge in interests \cite{ChiangZ2016}. This new technology, emerging from the rapid development of the Internet of Things, aims to reduce the burden of communication and computation at cloud or centralized servers by shifting the computing infrastructure closer to the source of data (e.g, smart devices, wireless sensors, or mobile robots). In this context, distributed algorithms provide a promising solution for coping with the large-scale complex networks while handling massive amounts of generated data.     

Distributed algorithms for these problems have received wide attention during the last decade, mostly focusing on three classes of algorithms, namely, the \textit{alternating direction method of multipliers (ADMM)} \cite{WeiO2012,BoydPCPE2011,ShiLYWY2014, MakhdoumiO2017}, \textit{distributed dual methods} (mirror descent/dual averaging) \cite{DuchiAW2012, TsianosLR2012, DoanBNB2019, LiCZW2016, YUANHHJ2018}, and \textit{distributed gradient algorithms} \cite{ShiLWY2015,GuLi2017,NedicO2009, NedicOS2017, YuanLY2016,LorenzoS2016,Doan2018,DoanBS2018, DoanBS2017,LanLZ2017,ShahB2018}.  The focus in this paper will be on distributed (sub)gradient algorithms, as they have the benefits (in terms of convergence rates and simplicity) of both ADMM and dual methods.  We refer interested readers to the recent survey paper \cite{NedicOR2018} for a summary of existing results in this area.

In distributed algorithms, the nodes are required to communicate and exchange information while cooperatively solving the problems. Thus, communication constraints, such as delays and finite bandwidth, are critical issues in distributed systems. For this reason, there are recent interests in studying the convergence of distributed gradient methods while taking into account these communication constraints. The convergence rates of such methods in the presence of communication delays have been studied in \cite{ WuYLYS2018, TianSDS2018,DoanBS2018, DoanBS2017}, while some works presented in \cite{LanLZ2017, NotarnicolaSSG2018} focus on reducing the number of communication rounds between nodes. 

Our focus is on studying the convergence properties of distributed gradient methods when the nodes are only allowed to exchange their quantized values due to the finite bandwidths shared between them. Different variants of distributed gradient methods under quantized communication have been studied in \cite{PuZJ2017, JueyouGZX2016, YiH2014, NedicOOT2008b, DoanMR2018, ReisizadehMHP2018}. In \cite{JueyouGZX2016, NedicOOT2008b} the authors only show the convergence to a neighborhood around the optimal of the problem due to the quantized error. On the other hand, an asymptotic convergence to the optimal has been studied in \cite{PuZJ2017,DoanMR2018,YiH2014}; however, a condition on the growing communication bandwidth is assumed in these works to remove the quantized error. Recently, the authors in \cite{ReisizadehMHP2018} study distributed gradient methods with random  quantization using finite bandwidths, and show a convergence rate $\mathcal{O}(1/k^{(1-\gamma)/2})$ for some $\gamma\in(0,1)$ for unconstrained problems with strongly convex and smooth objective functions.

We consider in this paper a stochastic variant of distributed gradient methods with random quantization, which can be viewed as a distributed version of the well-known two-time-scale stochastic approximation. Similar to \cite{ReisizadehMHP2018}, we consider the problems where the nodes only share a finite communication bandwidth. However, unlike \cite{ReisizadehMHP2018} we consider a constrained problem with nonsmooth objective functions. We derive explicit formulas for the rates of convergence of the algorithm, which show the dependence on the network topology and the communication capacity, for both convex and strongly convex objective functions. It is worth to note that the techniques used to derive the convergence rates in this paper are different from the ones in \cite{ReisizadehMHP2018}. While the authors in \cite{ReisizadehMHP2018} use a dual approach in their convergence analysis, we utilize the standard techniques from two-time-scale stochastic approximation studied in \cite{borkar2008,KY2009,WangFL2017,KondaT2004}. This allows us to clearly show the impact of network topology and communication bandwidths on the convergence of the algorithm.   

\textbf{Main Contributions}.
The main contributions of this paper are two folds. We first propose a distributed variant of the well-known two-time-scale stochastic approximation for solving network optimization problems under random quantization. Second, we study the convergence and derive upper bounds on the rates of convergence of such methods. In particular, when the objective function is convex we first show the almost sure convergence of the variables of the nodes to the optimal solution of the problem. Then under an appropriate choice of the step sizes, we derive the convergence of the objective function to the optimal value in expectation at a rate $\mathcal{O}\left(\Delta\ln(k+1)\,/\,(1-\sigma_2)^2k^{1/4}\right)$, where $k$ is the number of iterations and $1-\sigma_2$ represents for the connectivity of the underlying network. In addition, $\Delta$ represents for quantization errors, which depends on the size of the communication bandwidths. When the objective function is strongly convex, we further show that this rate occurs at $\mathcal{O}\left(C\ln(k+1)\,/\,(1-\sigma_2)^3k^{1/3}\right)$. We then conclude our paper with numerical experiments comparing the performance of distributed subgradient methods for solving the well-known regression problems with and without quantization.             

\subsection{Notation And Definition}
\textbf{Notation:} We first introduce here a set of notation and definition used throughout this paper. We use boldface to distinguish between vectors in $\Rset^d$ and scalars in $\Rset$. Given a collection of vectors $\xbf_1,\ldots,\xbf_n$ in $\Rset^d$, we denote by $\Xbf$ a matrix in $\Rset^{n\times d}$, whose $i$-th row is $\xbf_i^T$. We then denote by $\|\xbf\|$ and $\|\Xbf\|$ the Euclidean norm and the Frobenius norm of $\xbf$ and $\Xbf$, respectively.  Let $\1$ be the vector whose entries are $1$ and $\Ibf$ the identity matrix. Given a closed convex set $\Xcal$, we denote by $\Pcal_{\Xcal}[\xbf]$ the projection of $\xbf$ to $\Xcal$. 

Given a nonsmooth convex function $f:\Rset^{d}\rightarrow\Rset$, we denote by $\partial f(\xbf)$ its subdifferential estimated at $x$, i.e., $\partial f(\xbf) \triangleq \{g\in\Rset^{d}\,|\, f(\ybf) \geq f(x) + g^T(\ybf-\xbf), \; \forall \ybf\in\Rset^d \}$ is the set of subgradients of $f$ at $\xbf$. Since $f$ is convex, $\partial f(\cdot)$  is nonempty. The function $f$ is $L$-Lipschitz continuous if and only if 
\begin{align}
|\,f(\xbf)-f(\ybf)\,| \leq L\|\xbf-\ybf\|,\quad \forall\; \xbf,\ybf\in \Rset^d. \label{notation:Lipschitz}
\end{align}   
Note that the $L$-Lipschitz continuity of $f$ is equivalent to the subgradients of $f$ are uniformly bounded by $L$ \cite{ShalevShwartz2012}. A function $f$ is $\mu$-strongly convex if and only if $f$ satisfies $\forall x,y$
\begin{align}
f(y) - f(x) -g(x)^{T}(y-x)\geq \frac{\mu}{2}(y-x). \label{notation:sc}
\end{align}

\textbf{Random Quantization:} We now present a brief review of random quantization adopted from \cite{AysalCR2008}, which is also equivalent to dithered quantization in signal processing. In particular, given a finite interval $[\ell,u]$ we divide this interval into a $B$ number of bins $\ell = \tau_1\leq\tau_2\leq\ldots\leq\tau_B = u$. We assume that the points $\tau_i$ are uniformly spaced with a distance $\Delta$, i.e., $\Delta = \tau_{i+1}-\tau_i$ for all $i=0,\ldots,B-1$ implying that $\Delta = (u-\ell)\,/\,(B-1)$. Thus, to present the points $\tau_i$ we need a finite $b$ bits where $b = \log_2(B)$.  

Next given $x\in[\tau_i,\tau_i+1)$ we denote by $p = (x-\tau_i)\,/\,\Delta$. Then the random quantization $q$  of $x$ is defined as
\begin{align}
q = \Qcal(x) \triangleq \left\{\begin{array}{lll}
\tau_i & \text{w/ prob. } 1-p\\
\tau_{i+1} & \text{w/ prob. } p.
\end{array}\right.\label{notation:random_operator} 
\end{align}
As shown in \cite{AysalCR2008} the random quantization Eq.\ \eqref{notation:random_operator} satisfies
\begin{align}
\begin{aligned}
\Eset[q] = x,\qquad&\text{and}\qquad \Eset[(q-x)^2] \leq \frac{\Delta^2}{4}\\
& |\,x- q\,| \leq \Delta\quad {a.s.}
\end{aligned}\label{notation:random_quantization}
\end{align}
In addition, we have $q = \tau_i$ a.s. if $x = \tau_i$ for some $i = 1,\ldots,B.$ Thus, $q\in[\ell,u]$ a.s. if $x\in[\ell,u]$. 

Finally, we consider the random quantization for the vector case. In particular, consider a compact set $\Xcal\subset\Rset^d$ defined as
\begin{align*}
\Xcal = [\ell^1,u^1]\times\ldots\times[\ell^d,u^d].
\end{align*}
With some abuse of notation, given a vector $\xbf\in\Xcal$ we denote by $\qbf = \Qcal(\xbf)$, where $q^i = \Qcal(x^i)$, the quantization of $i$-th coordinate of $\xbf$, for $i = 1,2,\ldots,d$. Here, each $q^i$ is defined by using Eq.\ \eqref{notation:random_operator} with a uniform distance $\Delta^i$ associated with each interval $[\ell^i,u^i]$, for all $i=1\ldots,d$.

\section{Problem Formulation}\label{sec:ProbForm}
We consider an optimization problem defined over a network of $n$ nodes. Associated with each node $i$ is a nonsmooth convex function $f_i:\Xcal\rightarrow\Rset$ over a convex compact set $\Xcal\subset\Rset^d$. The goal is to solve  
\begin{align}
\underset{\xbf\in\Xcal}{\text{minimize }} f(\xbf)\triangleq \sum_{i=1}^n f_i(\xbf).\label{prob:obj}
\end{align}
Each node $i$ knows only its local function $f_i$, and since there is no central coordination, the nodes are required to cooperatively solve the problem. We will use a distributed consensus-based (sub)gradient method where each node $i$ maintains their own version of the decision variables $\xbf_i\in\Rset^{d}$; the goal is to have all the $\xbf_i$ converge to $\xbf^*$, a solution of problem \eqref{prob:obj}. Each node can exchange a quantized version of $\xbf_i$ with its neighbors, as defined through a connected and undirected graph $\Gcal = (\Vcal,\Ecal)$, where $\Vcal = \{1,\ldots,n\}$ and  $\Ecal = (\Vcal\times\Vcal)$ are the vertex and edge sets, respectively. We denote by $\Ncal:= \{ j \in \Vcal\; |\; (i, j) \in\Ecal\}$ the set of node $i$'s neighbors. 
%The goal of the nodes is to asymptotically drive the nodes' estimates $\xbf_i$ to $\xbf^*$. 

A concrete motivating example for this problem is distributed linear regression problems solved over a network of processors. Regression problems involving massive amounts of data are common in machine learning; see for example, \cite{SSBD2014, HTF2009}. Each function $f_i$ is the empirical loss over the local data stored at processor $i$. The objective is to minimize the total loss over the entire dataset. Due to the difficulty of storing the enormous amount of data at a central location, the processors perform local computations over the local data, which are then exchanged to arrive at the globally optimal solution. Distributed gradient methods are a natural choice to solve such problems since they have been observed to be both fast and easily parallelizable in the case where the processors can exchange data instantaneously. The goal of this paper is to show that the algorithm continues to be convergent even the nodes only exchange the quantized values of their variables due to the finite bandwidths shared between them. In particular, we derive expressions for the convergence rate as a function of the communication bandwidths and the underlying network topology. 

In the sequel, we will use $f^*$ to denote the optimal value of problem \eqref{prob:obj}, i.e., $f^* =\sum_{i=1}^n f_i(\xbf^*)$ where $\xbf^*$ is a solution of problem \eqref{prob:obj}. We denote by $\Xcal^*$ the solution set of problem \eqref{prob:obj}, which is nonempty due to the compactness of $\Xcal$. In addition, since $\Xcal$ is compact it is obvious that each $f_i$ is Lipschitz continuous with some positive constant $L_i$, as stated in the following proposition. 
\begin{prop}\label{prop:bounded_subg}
Each function $f_i$, for all $i\in\Vcal$, is $L_i$-Lipschitz continuous, i.e., Eq.\ \eqref{notation:Lipschitz} holds for some $L_i\geq0$ for all $i\in\Vcal$.
\end{prop}
%Finally, for ease of exposition and notational convenience, we only consider problem \eqref{prob:obj} when the variable $x$ is a scalar. We note that, however, extensions for the multi-dimensional case can easily be derived from the results in this paper. 

%!TEX root = random_quantization.tex

\section{Distributed Gradient Methods Under Random Quantization}\label{sec:dither_quantization}
Distributed subgradient ({\sf DSG}) methods, Eq.\ \eqref{distributed:DSGD}, for solving problem \eqref{prob:obj} were first studied and analyzed rigorously in \cite{NedicO2009,NedicOP2010}. In these methods each node $i$ iteratively updates $\xbf_i$ as 
\begin{align}
\xbf_i(k+1) = \Pcal_{\Xcal}\left[\sum_{j\in\Ncal_i} a_{ij}\xbf_j(k) \; - \; \alpha(k) \gbf_i(\xbf_i(k))\right],\label{distributed:DSGD}
\end{align}  
where $\alpha(k)$ is some sequence of stepsizes and $\gbf_i(\xbf_i(k))\in\partial f_i(\xbf_i(k))$. Here, $a_{ij}$ is some positive weight which node $i$ assigns for $\xbf_j$. We assume that these weights, which capture the topology of $\Gcal$, satisfy the following condition. 
\begin{assump}\label{assump:doub_stoch}  
The matrix $\Abf$, whose $(i,j)$-th entries are $a_{ij}$, is doubly stochastic, i.e., $\sum_{i=1}^n a_{ij} =
\sum_{j=1}^n a_{ij} = 1$. Moreover, $\Abf$ is irreducible and aperiodic. Finally, the weights $a_{ij} > 0$ if and only if $(i, j) \in \Ecal$ otherwise $a_{ij} = 0$.
\end{assump} 
This assumption also implies that $\Abf$ has $1$ as the largest singular value and others are strictly less than $1$; see for example, the Perron-Frobenius theorem \cite{HJ1985}.  Also, we denote by $\sigma_2\in(0,1)$ the second largest singular value of $\Abf$, which by the Courant-Fisher theorem \cite{HJ1985} gives
\begin{align}
\left\|\Abf\left(\Ibf-\frac{1}{n}\1\1^T\right)\right\| \leq \sigma_2\left\|\Ibf-\frac{1}{n}\1\1^T\right\|.\label{const:sigma_2}
\end{align}

Our focus in this section is to study {\sf DSG} under random quantization in communication between nodes. In particular, at any iteration $k\geq0$ the nodes are only allowed to send and receive the quantized values of their local copies to their neighboring nodes. Due to such quantized communication, we modify the update in Eq.\ \eqref{distributed:DSGD}, that is, each node $i$ now considers the following update
\begin{align*}
\xbf_i(k+1) &= \Pcal_{\Xcal}\Bigg[ (1-\beta(k))\xbf_i(k) +\beta(k)\sum_{j\in\Ncal_i} a_{ij}\qbf_j(k)\notag\\
&\qquad\qquad - \alpha(k)  \gbf_i(\xbf_i(k))\Bigg],
%\label{distributed:QDSGD}
\end{align*}
where $\qbf_j(k) = \Qcal(\xbf_j(k))$ given in Eq.\ \eqref{notation:random_operator}. Here, in addition to $\alpha(k)$ we introduce a new stepsize $\beta(k)$ due to the random quantization exchanged between nodes. 

This update has a simple interpretation. At any time $k\geq0$, each node $i$ first obtains the quantized value $q_i(k)$ of its value $x_i(k)$. Each node $i$ then formulates a convex combination between its value $x_i(k)$ and the weighted quantized value received from its neighbors $j\in\Ncal_j$, with the goal of seeking a consensus on their estimates. Each node then moves along the subgradients of its respective objective function to update its estimates, pushing the consensus point toward the optimal set $\Xcal^*$. The distributed subgradient algorithm under random quantization is formally stated in Algorithm \ref{alg:QDSG}.  

\subsection{The Role of $\beta(k)$}
We discuss in this section some aspects of the new stepsize $\beta(k)$. First, one can interpret Eq.\ \eqref{alg:xiUpdate} as a distributed two-time-scale stochastic algorithm \cite{borkar2008,KY2009,WangFL2017,KondaT2004}, where the first sum play the role of fast time scale while the gradient step is the slow time scale. Due to the random quantization, each node first uses the fast time scale to estimate the true average of their estimates. Each node then applies the gradient step to slowly push its estimate toward a solution of problem \eqref{prob:obj}. As will be seen, $\beta(k)$ will be chosen relatively larger as compared to $\alpha(k)$ to guarantee for the convergence of Algorithm \ref{alg:QDSG}. 

Second, it is obvious that when $\beta(k) = 1$, for all $k$, we recover the update in Eq.\ \eqref{distributed:DSGD}. In a sense, introducing $\beta(k)$ gives us one more freedom to design our algorithm, especially when dealing with communication constraints. This has also been observed in our previous works \cite{DoanBS2017,DoanBS2018} where we use a constant $\beta$ to study the impact of network latencies on the performance of distributed gradient methods. 

Finally, one can view $\beta(k)$, in addition to $a_{ij}$, is some weight which each node $i$ uses to indicate that it ``trusts" its own value $x_i$ more than the value $x_j$ received from its neighbor $j$. As will be seen, to guarantee the convergence of the algorithm we will let $\beta(k)$ go to zero at some proper rate, implying eventually node $i$ only uses its own value. 

%\RestyleAlgo{boxruled}
\begin{algorithm}[t]
\caption{Distributed Subgradient Algorithm\quad Under Random Quantization}
1. \textbf{Initialize}: Each node $i$ arbitrarily initializes $\xbf_i(0)\in\Xcal$.\\
\smallskip
2. \textbf{Iteration}: For $k\geq 0$ each node $i$ implements
\smallskip
\begin{flalign}
\xbf_i(k+1) = \Pcal_{\Xcal}\Big[&
(1-\beta(k))\xbf_i(k) +\beta(k)\sum_{j\in\Ncal_i} a_{ij}\qbf_j(k)\notag\\
&\qquad - \alpha(k) \gbf_i(\xbf_i(k))\Big].
\label{alg:xiUpdate}
\end{flalign}
\label{alg:QDSG}
\end{algorithm}

%!TEX root = random_quantization.tex

\section{Convergence Results}\label{sec:analysis}
The focus of this section is to study the convergence properties of Algorithm \ref{alg:QDSG} for solving problem \eqref{prob:obj}, when the objective functions are both convex and strongly convex.  The key idea of our analysis is to utilize the standard techniques used in centralized subgradient methods and stochastic approximation approach. In particular, for convex objective functions we first show that $\xbf_i(k)$, for all $i\in\Vcal$, converges almost surely to a solution $x^*$ of problem \eqref{prob:obj} under some proper choice of stepsizes $\{\alpha(k),\beta(k)\}$. We next show the convergence of the function $f$ estimated at the time $\alpha-$weighted average of each $x_i$ to the optimal value $f^*$ in expectation at a rate $\mathcal{O}\left(n\Delta^2 L^2\ln(k)\,/\,(1-\sigma_2)^2k^{1/4}\right)$, where $1-\sigma_2$ is the spectral gap of the network connectivity. Finally, when the objective functions are strongly convex, we derive the convergence of the time-weighted average of each $\xbf_i$ to an optimal solution $\xbf^*$ of problem \eqref{prob:obj} in expectation at a rate $\mathcal{O}\left(n\Delta^2 L^2\ln(k)\,/\,(1-\sigma_2)^3k^{1/3}\right)$.  

We start our analysis by introducing more notation. Given the nodes' estimates $\xbf_1,\ldots,\xbf_n$ in $\Rset^{d}$ we denote by $\Xbf\in\Rset^{n\times d}$ a matrix whose $i$-th rows are $\xbf_i^T$, i.e.,
\begin{align*}
\Xbf = \left[\begin{array}{cc}
-\;\xbf_1^T\;-  \\
\cdots\\
-\;\xbf_n^T\;-
\end{array} \right]\in\Rset^{n\times d}\cdot    
\end{align*}
Let $\xbarbf\in\Rset^{d}$ be the average of $\xbf_i$, i.e., 
\begin{align*}
\xbarbf = \frac{1}{n}\sum_{i=1}^n\xbf_i = \Xbf^T\1 \in\Rset^{d}.    
\end{align*}
For convenience, we use the following notation
\begin{align*}
&\Gbf(\Xbf) =\left[\begin{array}{cc}
-\;\gbf_1^T(\xbf_1)\;-  \\
\cdots\\
-\;\gbf_n^T(\xbf_n)\;-
\end{array} \right]\in\Rset^{n\times d},\quad \ebf_i(k) = \qbf_i(k) - \xbf_i(k)\notag\\
&\Delta = \sum_{\ell=1}^d\Delta^{\ell},\quad L = \sum_{i=1}^nL_i,\quad \Wbf = \Ibf - \frac{1}{n}\1\1^T\\
&r(k) = \|\xbarbf(k) - \xbf^*\|^2,\quad\Ybf(k)=\Xbf(k)-\1\xbarbf^T(k)= \Wbf\Xbf(k).
\end{align*}
Moreover, let $\Fcal_k$ be the filtration containing all the history generated by Eq.\ \eqref{alg:xiUpdate} upto time $k$, i.e., $$\Fcal_k = \{\Xbf(0),\Qbf(0),\ldots,\Xbf(k),\Qbf(k)\}.$$ 
Finally, given a vector $\vbf\in\Rset^{d}$ let $\xibf$ denote the error due to the projection of $\vbf$ to $\Xcal$, i.e., 
\begin{align}
\xibf =  \vbf - \Pcal_{\Xcal}[\vbf].\label{notation:xi}
\end{align}
Thus, Eq.\ \eqref{alg:xiUpdate} now can be rewritten as
\begin{align}
\begin{aligned}
\vbf_i(k) &= (1-\beta(k))\xbf_i(k) + \beta(k)\sum_{j\in\Ncal_i}a_{ij}\qbf_j(k)\\ 
&\qquad\qquad- \alpha(k)\gbf_i(\xbf_i(k))\\
\xbf_i(k+1) &= \vbf_i(k) - \xibf_i(k),
\end{aligned}\label{analysis:xiUpdate}
\end{align}
which by using $\Abf$ the matrix form of Eq.\ \eqref{analysis:xiUpdate} is given as
\begin{align}
\begin{aligned}
\Vbf(k) &= (1-\beta(k))\Xbf(k) + \beta(k)\Abf\Qbf(k)- \alpha(k) \Gbf(\Xbf(k))\\ 
\Xbf(k+1) &= \Vbf(k) - \Xibf(k),
\end{aligned}\label{anlaysis:Xupdate}
\end{align}
where $\Xibf(k)\in\Rset^{n\times d}$ is the matrix whose $i$-th row is $\xibf_i^T(k)$.  In addition, since $\Abf$ is doubly stochastic, we have 
\begin{align}
\begin{aligned}
\vbarbf(k+1) &= (1-\beta(k))\xbarbf(k)  + \beta(k)\qbarbf(k)\\
&\qquad\qquad-\frac{\alpha(k)}{n}\sum_{i=1}^n \gbf_i(\xbf_i(k))\\ 
\xbarbf(k+1) &= \vbarbf(k) - \xibarbf(k).
\end{aligned}\label{analysis:xbar}
\end{align}

\subsection{Preliminaries}
\label{sec:preliminaries}
In this section, we consider some preliminary results, which are essential in our analysis given in the next section. For an ease of exposition, we delay the proofs of all results in this section to the appendix. However, we present a sketch of their proofs to explain some intuition behind our analysis 

We first provide an upper bound for the consensus error $\|\Ybf(k)\| = \|\Wbf\Xbf(k)\|$ in the following lemma. 

\begin{lem}\label{lem:consensus_bound}
Suppose that Assumption \ref{assump:doub_stoch} holds. Let the sequence $\{x_i(k)\}$, for all $i\in\Vcal$, be generated by Algorithm \ref{alg:QDSG}. In addition, let $\{\alpha(k),\beta(k)\}$ be two sequences of nonnegative and nonincreasing stepsizes. Then we have
\begin{align}
&\Eset[\|\Ybf(k+1)\|^2\,|\,\Fcal_k]\notag\\
&\quad \leq  (1-(1-\sigma_2)\beta(k))\|\Ybf(k)\|^2 \notag\\
&\quad\qquad  + n\sigma_2^2\Delta^2\beta^2(k)+\frac{4L^2(\beta(0)+1)}{(1-\sigma_2)}\frac{\alpha^2(k)}{\beta(k)}\cdot\label{lem_consensus:UpperBound}
\end{align}
Moreover, we also obtain
\begin{align}
&\sum_{t=0}^{k}\beta(t)\Eset[\,\|\Ybf(t)\|^2\,]\notag\\
&\leq \frac{\Eset[\,\|\Ybf(0)\|^2\,]}{1-\sigma_2}+\sum_{t=0}^{k}\frac{n\sigma_2^2\Delta^2\beta^2(t)}{1-\sigma_2} + \frac{4L^2(\beta(0)+1)\alpha^2(t)}{(1-\sigma_2)^2\beta(t)}\cdot\label{lem_consensus:UpperBound_rate}
\end{align}
\end{lem}
\begin{proof}[Sketch of Proof]
To show Eq.\ \eqref{lem_consensus:UpperBound} we first use Eqs.\ \eqref{anlaysis:Xupdate} and \eqref{analysis:xbar} to have
\begin{align*}
\Ybf(k+1) &= \Wbf\Xbf(k+1)\notag\\
&= (1-\beta(k))\Ybf(k) + \beta(k)\Abf\Wbf\Qbf(k) -\alpha(k)\Wbf\Gbf(k) - \Wbf\Xibf(k).
\end{align*}
Next, we use the following Cauchy-Schwarz inequality with some $\eta > 0$ and $a,b\in\Rset$   
\begin{align*}
(a+b)^2\leq (1+\eta)a^2 + (1+1/\eta)b^2.
\end{align*}
Thus, by taking the $2$-norm square of the first equation and using the preceding Cauchy-Schwarz inequality with $\eta = 1 + (1-\sigma_2)\beta(k)$ we have
\begin{align}
&\|\Ybf(k+1)\|^2\notag\\ 
&\leq (1+(1-\sigma_2)\beta(k)) \left\|(1-\beta(k))\Ybf(k) + \beta(k)\Abf\Wbf\Qbf(k)\right\|^2\notag\\
&\qquad + \left(2+\frac{2}{(1-\sigma_2)\beta(k)}\right)\left\|\alpha(k)\Wbf\Gbf(k)\right\|^2\notag\\
&\qquad + \left(2+\frac{2}{(1-\sigma_2)\beta(k)}\right)\left\|\Wbf\Xibf(k)\right\|^2,\label{lem_consensus_analysis:Eq1}
\end{align}
We next analyze each term on the right-hand side of Eq.\ \eqref{lem_consensus_analysis:Eq1}. First, Proposition \ref{prop:bounded_subg} gives
\begin{align*}
\left\|\Wbf\Gbf(k)\right\|^2 \leq \|\Gbf(k)\|^2 \leq L^2.
\end{align*}
Second, using the projection lemma, Lemma \ref{apx_lem:projection}(b) in the Appendix, one can show
\begin{align*}
&\|\Wbf\Xibf(k)\|^2 \leq L^2\alpha^2(k). 
\end{align*}
Third, by Eq.\ \eqref{notation:random_quantization} we have
\begin{align*}
\|\Ebf(k)\|^2 &= \sum_{i=1}^{n}\|\xbf_i(k)-\qbf_i(k)\|^2\leq \sum_{i=1}^{n}\Delta^2 = n\Delta^2.
\end{align*}
Fourth, by Eq.\ \eqref{const:sigma_2} we have
\begin{align*}
&\|(1-\beta(k))\Ybf(k) + \beta(k)\Abf\Ybf(k)\|^2\notag\\
&\qquad \leq \|(1-(1-\sigma_2)\beta(k))\Ybf(k)\|.    \end{align*}
Thus, taking the conditional expectation of Eq.\ \eqref{lem_consensus:Eq2} w.r.t. $\Fcal_k$ and using the last four inequalities we obtain Eq.\ \eqref{lem_consensus:UpperBound}.

Finally, taking the expectation on both sides of Eq.\ \eqref{lem_consensus:UpperBound} and summing up over $k=0,\ldots,K$ for some $K$ immediately give Eq.\ \eqref{lem_consensus:UpperBound_rate}. 
\end{proof}

We next provide proper conditions on the stepsizes $\{\alpha(k),\beta(k)\}$, which guarantees that the nodes achieve a consensus on their estimates $\xbf_i$. The analysis of this lemma is a consequence of Lemma \ref{lem:consensus_bound} and Lemma \ref{lem:Martingale} on the almost supermartingale convergence theorem given later.    
\begin{lem}\label{lem:consensus}
Suppose that Assumption \ref{assump:doub_stoch} holds. Let the sequence $\{x_i(k)\}$, for all $i\in\Vcal$, be generated by Algorithm \ref{alg:QDSG}. In addition, let $\alpha(k)$ and $\beta(k)$ satisfy 
\begin{align}
\begin{aligned}
&\sum_{k=0}^{\infty} \alpha(k) =\sum_{k=0}^{\infty} \beta(k) = \infty,\quad\sum_{k=0}^{\infty} \frac{\alpha^2(k)}{\beta(k)} < \infty\\
&\sum_{k=0}^{\infty} \alpha^2(k) < \infty,\quad \sum_{k=0}^{\infty} \beta^2(k) < \infty. 
\end{aligned}\label{lem_consensus:stepsizes}
\end{align}
Then we have 
\begin{align}
\lim_{k\rightarrow\infty}\|\,\xbf_i(k)-\xbarbf(k)\,\| = 0\quad \text{a.s.,}\quad\text{for all}\; i\in\Vcal.\label{lem_consensus:AsympCov}
\end{align}
Furthermore, the following condition holds
\begin{align}
\sum_{k=0}^{\infty}\beta(k)\|\Ybf(k)\|^2 < \infty\qquad\text{a.s.} \label{lem_consensus:FiniteSum}
\end{align}
\end{lem}

\begin{remark}
One example of stepsizes $\{\alpha(k),\beta(k)\}$, which satisfies Eq.\ \eqref{lem_consensus:stepsizes}, can be chosen as follows
\begin{align}
\alpha(k) = \frac{1}{k+2},\quad \beta(k) = \frac{1}{(k+2)^s},\; \forall s\in\left(\frac{1}{2}\,,\,1\right).\label{stepsizes:example}
\end{align}
\end{remark}
Third, we study an upper bound for the optimal distance $\rbf(k) = \|\xbar(k)-\xbf^*\|^2$ in the following lemma. 
\begin{lem}\label{lem:opt_dist}
Suppose that Assumption \ref{assump:doub_stoch} holds. Let the sequence $\{x_i(k)\}$, for all $i\in\Vcal$, be generated by Algorithm \ref{alg:QDSG}. Let $\{\alpha(k),\beta(k)\}$ be two sequences of nonnegative and nonincreasing stepsizes with $\beta(0)<1$. Let $\xbf^*$ be a solution of problem \eqref{prob:obj}. Then we have
\begin{align}
& \Eset\left[\,r(k+1)\,|\,\Fcal_k\,\right]\notag\\ 
&\leq r(k) + \frac{6L^2\alpha^2(k)}{n(1-\beta(0))}  +\frac{2L^2\alpha^2(k)}{n\beta(k)} + \Delta^2\beta^2(k)\notag\\
&\qquad + \frac{2\beta(k)\|\Ybf(k)\|^2}{n}- \frac{2\alpha(k)}{n}\sum_{i=1}^n \gbf_i^T(x_i(k))\Big(\xbf_i(k)-\xbf^*\Big).\label{lem_opt_dist:Ineq}
\end{align}
\end{lem}
\begin{proof}[Sketch of Proof]
First, using Eq.\ \eqref{analysis:xbar} to have
\begin{align*}
&r(k+1) = \|\vbarbf(k)-\xibarbf(k)-\xbf^*\|^2\notag\\ 
&= \left\|\begin{array}{ll} (1-\beta(k))\xbarbf(k)-\xbf^* -\frac{\alpha(k)}{n}\sum_{i=1}^n \gbf_i(\xbf_i(k))\\ 
\qquad\qquad + \beta(k)\qbarbf(k)-\xibarbf(k)
\end{array}
\right\|^2,
\end{align*}
which by expanding the right-hand side and taking the conditional expectation w.r.t $\Fcal_k$ yields
\begin{align*}
&\Eset\left[\,r(k+1)\,|\Fcal_k\,\right]\notag\\
&= \left\|\xbarbf(k)-\xbf^*-\frac{\alpha(k)}{n}\sum_{i=1}^n \gbf_i(x_i(k))\right\|^2\notag\\ 
&\quad + \left\|\beta(k)\ebarbf(k)\right\|^2 + \left\|\xibarbf(k)\right\|^2  - 2\xibarbf^T(k)\Big(\xbarbf(k)-\xbf^*\Big)\notag\\ 
&\qquad + \frac{2\alpha(k)}{n}\xibarbf^T(k) \sum_{i=1}^n \gbf_i(\xbf_i(k)).
\end{align*}
The next step is to provide an upper bound for each term on the right-hand side of the preceding equation to obtain Eq.\ \eqref{lem_opt_dist:Ineq}. This step can be done in a similar concept of the one given in Lemma \ref{lem:consensus_bound}. 
\end{proof}
Finally, we utilize the result on almost supermartingale convergence studied in \cite{Robbins1971}, stated as follows.
\begin{lem}[\cite{Robbins1971}]\label{lem:Martingale}
Let $\{y(k)\},$ $\{z(k)\}$, $\{w(k)\}$, and $\{\gamma(k)\}$ be non-negative sequences of random variables and satisfy 
\begin{align*}
&\mathbb{E}\Big[\,y(k+1)\,|\,\mathcal{F}_{k}\,\Big] \leq (1+\gamma(k))y(k) - z(k) + w(k)\\
&\sum_{k=0}^\infty \gamma(k) < \infty \text{ a.s, }\quad \sum_{k=0}^\infty w(k) < \infty \text{ a.s},
\end{align*}
where $\mathcal{F}_k = \{y(0),\ldots,y(k)\}$, the history of $y$ up to time $k$. Then $\{y(k)\}$ converges a.s., and  $\sum_{k=0}^{\infty} z(k)<\infty$ a.s.
\end{lem}

\subsection{Convergence Results of Convex Functions}
In this section, we study the convergence and the rate of convergence of Algorithm \ref{alg:QDSG} when the objective functions $f_i$ are convex. For an ease of explanation, we only provide a sketch of the proofs for all the main results in this section and the next section, where their details are presented in Section \ref{sec:proof_main_results}.  

Our first main result is to show that if the stepsizes $\{\alpha(k),\beta(k)\}$ satisfy Eq.\ \eqref{lem_consensus:stepsizes}, then $\xbf_i(k)$, for all $i\in\Vcal$, converges almost surely to $\xbf^*$, a solution of problem \eqref{prob:obj}. The following theorem is states this result.  
\begin{thm}\label{thm:asymp_conv}
Suppose that Assumption \ref{assump:doub_stoch} holds. Let the sequence $\{\xbf_i(k)\}$, for all $i\in\Vcal$, be generated by Algorithm \ref{alg:QDSG}. Let $\{\alpha(k),\beta(k)\}$ be two sequences of nonnegative and nonincreasing stepsizes satisfying Eq.\ \eqref{lem_consensus:stepsizes} with $\beta(0)<1$, e.g., Eq.\ \eqref{stepsizes:example} holds. 
Then we have
\begin{align}
\lim_{k\rightarrow\infty} \xbf_i(k)  = \widetilde{\xbf} \qquad\text{a.s.,}\qquad \text{for all } i\in\Vcal,\label{thm_opt:asympconv} 
\end{align}
for some  $\widetilde{\xbf}$ that is a solution of Problem \eqref{prob:obj}.
\end{thm}

\begin{proof}[Proof Sketch]
The main idea of this proof is first using the convexity of the functions $f_i$ into Eq.\ \eqref{lem_opt_dist:Ineq} in Lemma \ref{lem:opt_dist} to obtain 
\begin{align}
\Eset[r(k+1)\,|\,\Fcal_k] &\leq r(k) + \mathcal{O}\left(\alpha^2(k) + \beta^2(k) + \frac{\alpha^2(k)}{\beta(k)}\right)\notag\\    
&\qquad + \frac{3}{n} \beta(k)\|\Ybf(k)\|^2 -\frac{2\alpha(k)}{n}\Big(f(\xbarbf(k))-f^*\Big).\label{thm_opt:Eq1_Sketch} 
\end{align}
Second, since the stepsizes $\{\alpha(k),\,\beta(k)\}$ satisfy the conditions in Eq.\ \eqref{lem_consensus:stepsizes}, Eq.\ \eqref{lem_consensus:FiniteSum} holds. Thus, we can now apply Lemma \ref{lem:Martingale} to the preceding relation to have
\begin{align*}
&\lim_{k\rightarrow\infty}r(k)\qquad \text{ exits a.s. for each } \xbf^*\\ &\sum_{k=0}^\infty\alpha(k)\left(f(\xbarbf(k))  - f^*\right) < \infty\qquad \text{a.s.}
\end{align*}
Thus, using these relations and standard analysis on the convergence of subsequence of $\{\xbarbf(k)\}$ we can obtain Eq.\ \eqref{thm_opt:asympconv}.   
\end{proof}
We now study the rate of convergence of Algorithm \ref{alg:QDSG} to the optimal value in expectation, where  we utilize a similar technique used to establish the convergence rate of centralized subgradient methods. In particular, we show that if each node $i$ maintains a variable $\zbf_i$ used to estimate the time $\alpha-$weighted average of its local copy $\xbf_i$, then the function value $f$ estimated at each $\zbf_i$ converges in expectation to the optimal value $f^*$
with a rate $\mathcal{O}\left(n\Delta^2L^2\ln(k)\,/\,(1-\sigma_2)^2)k^{1/4}\right)$. The dependence on the variance $\Delta^2$ of the quantized error in the upper bound of the rate is natural, as we often observe in stochastic gradient descent where such dependence is on the variance of the gradient noise. Such result is derived under different assumptions on the stepsizes $\{\alpha(k),\beta(k)\}$ as shown in the following theorem\footnote{We note that the conditions on the stepsizes in Theorems \ref{thm:asymp_conv} and \ref{thm:rate_conv} are common choices to derive the asymptotic convergence and the rate of centralized subgradient methods, respectively; see for example \cite{Nesterov2004}.}. Note that while the previous theorem studies almost sure convergence of the local copies, this theorem studies convergence in expectation of the functional value, and so it is not surprising that the stepsizes are different.

\begin{thm}\label{thm:rate_conv}
Suppose that Assumption \ref{assump:doub_stoch} holds. Let the sequence $\{x_i(k)\}$, for all $i\in\Vcal$, be generated by Algorithm \ref{alg:QDSG}. Let  $\{\alpha(k),\beta(k)\}$ be defined as 
\begin{align}
\alpha(k) = \frac{1}{(k+2)^{3/4}},\qquad\beta(k) = \frac{1}{(k+2)^{1/2}}\cdot\label{thm_rate_conv:stepsizes}
\end{align}
In addition, suppose that each node $i$ maintains a variable $z_i$ initialized arbitrarily in $\Xcal$ and updated as 
\begin{align*}
\zbf_i(k) = \frac{\sum_{t=0}^{k}\alpha(t)\xbf_i(t)}{\sum_{t=0}^{k}\alpha(t)}\cdot
\end{align*}
Then we have for all $i\in\Vcal$ and $k\geq0$
\begin{align}
&\Eset\left[\,f(\zbf_i(k))\,\right] - f^*\notag\\
&\leq \left\{
\begin{array}{ll}
\frac{n\Eset\left[\,r(0)\,\right]}{8}  +\frac{n\Eset[\,\|\Ybf(0)\|^2\,]}{2(1-\sigma_2)}+ 16L^2\\
+\frac{9nL^2(1+\ln(K+2))}{2(1-\sigma)^2}+\frac{5n^2\Delta^2(1+\ln(K+2)}{8}
\end{array}\right\}\times \frac{1}{(k+1)^{1/4}}\cdot\label{thm_rate_conv:Ineq} 
\end{align}
\end{thm}

\begin{proof}[Proof Sketch]
The analysis of this theorem is divided into there main steps. First, we fix some $\ell\in\Vcal$, and utilize Eq.\ \eqref{thm_opt:Eq1_Sketch} and Proposition \ref{prop:bounded_subg} to have   
\begin{align*}
\Eset[r(k+1)\,|\,\Fcal_k] &\leq r(k) + \mathcal{O}\left(\alpha^2(k) + \beta^2(k) + \frac{\alpha^2(k)}{\beta(k)}\right)\\    
&\qquad + \frac{4}{n} \beta(k)\|\Ybf(k)\|^2 -\frac{2\alpha(k)}{n}\Big(f(\xbf_{\ell}(k))-f^*\Big).    
\end{align*}
Second, we utilize Eq.\ \eqref{lem_consensus:UpperBound_rate} to obtain the following for some $K>0$
\begin{align*}
\sum_{k=0}^{K}\beta(k)\Eset[\,\|\Ybf(k)\|^2\,] \leq \mathcal{O}\left(\sum_{k=0}^{K}\beta^2(k)+\frac{\alpha^2(k)}{\beta(k)}\right).   
\end{align*}
Third, using the preceding equation and the integral test with $\alpha(k),\beta(k)$ in Eq.\ \eqref{thm_rate_conv:stepsizes} into step $1$ with some algebraic manipulation immmediately gives us Eq.\ \eqref{thm_rate_conv:Ineq}.  
\end{proof}

\begin{remark}
We note that $\zbf_i(k)$ in Theorem \ref{thm:rate_conv} can be iteratively updated as follows
\begin{align*}
\zbf_i(k+1) = \frac{\alpha(k)\xbf_i(k)+S(k)\zbf_i(k)}{S(k+1)},
\end{align*}
where $S(0) = 0$ and $S(k+1) = \sum_{t=0}^{k}\alpha(t)$ for $k\geq1$.
\end{remark}

\subsection{Convergence Results of Strongly Convex Functions}
We study here the convergence rate of Algorithm \ref{alg:QDSG} when $f_i$ are strongly convex, that is, we consider the following assumption.
\begin{assump}\label{assump:sc}
Each function $f_i$, for all $i\in\Vcal$, is $\mu_i$-strongly convex, i.e., Eq.\ \eqref{notation:sc} holds for some $\mu_i\geq0$.
\end{assump}
Note that this assumption implies that $f$ is $\mu-$strongly convex where $\mu = \min_{i}\mu_i$. Under this assumption, we show the rate of convergence of Algorithm \ref{alg:QDSG} to an optimal solution $\xbf^*$ of problem \eqref{prob:obj} in expectation, stated in the following theorem. 
\begin{thm}\label{thm_rate_conv_sc}
Suppose that Assumptions \ref{assump:doub_stoch} and \ref{assump:sc} hold. Let the sequence $\{x_i(k)\}$, for all $i\in\Vcal$, be generated by Algorithm \ref{alg:QDSG}. Let $\xbf^*$ be a solution of problem \eqref{prob:obj} and $\{\alpha(k),\beta(k)\}$ be defined as 
\begin{align}
\begin{aligned}
&\alpha(k) = \frac{a}{k+2}\qquad\qquad \text{for }a \geq \frac{1}{\mu}\\
&\beta(k) = \frac{b}{(k+2)^{2/3}}\qquad \text{for } b\geq \frac{1}{1-\sigma_2}\cdot
\end{aligned}\label{thm_rate_conv_sc:stepsizes}
\end{align}
In addition, suppose that each node $i$ maintains a variable $\zbf_i$ initialized arbitrarily and updated as 
\begin{align*}
\zbf_i(k) = \frac{\sum_{t=0}^{k}\xbf_i(t)}{k+1}\cdot
\end{align*}
Then we have for all $i\in\Vcal$ and $k\geq0$
\begin{align}
&\Eset \Big[\,\|\zbf_{i}(k) - \xbf^*\|^2\Big]\notag\\
&\leq \left\{\begin{array}{ll}
n\Eset[\,r(0)\,] + \frac{4n\Eset[\,\|\Ybf(0)\|\,](1+\ln(k+2))}{(1-\sigma_2)}\\
+ \frac{6L^2\alpha^2(0)(1+\ln(k+2))}{1-\beta(0)}
\end{array}\right\}\times\frac{1}{k+2} \notag\\
&\quad +\left\{\begin{array}{ll}
n\Delta^2\beta^2(0) +  \frac{4L^2\alpha^2(0)}{\beta(0)\mu} \\
+ \frac{8n^2\sigma_2^2\Delta^2\beta^2(0)}{(1-\sigma_2)^2} + \frac{27nL^2\alpha(0)}{\beta(0)(1-\sigma_2)^3}
\end{array}\right\}\times\frac{1}{(k+2)^{1/3}}\cdot\label{thm_rate_conv_sc:Ineq}
\end{align}
\end{thm}

\begin{proof}[Proof Sketch]
The first step in this analysis is using the strong convexity of the functions $f_i$ into Eq.\ \eqref{lem_opt_dist:Ineq} to have
\begin{align*}
\Eset[r(k+1)\,|\,\Fcal_k] &\leq (1-\mu\alpha(k))r(k) + \mathcal{O}\left(\alpha^2(k) + \beta^2(k) + \frac{\alpha^2(k)}{\beta(k)}\right)\\    
&\qquad + \frac{3}{n} \beta(k)\|\Ybf(k)\|^2 -\frac{2\alpha(k)}{n}\Big(f(\xbarbf(k))-f^*\Big).
\end{align*} 
The rest of this proof is similar to the one in Theorem \ref{thm:rate_conv}.
\end{proof}

%!TEX root = random_quantization.tex

\section{Simulations}\label{sec:sim}
In this section, we apply Algorithm \ref{alg:QDSG} for solving linear regression problems, the most popular technique for data fitting \cite{HTF2009,SSBD2014} in statistical machine learning, over a network of processors under random quantization. The goal of this problem is to find a linear relationship between a set of variables and some real value outcome. That is, given a training set $S = \{(\abf_i,b_i)\in\mathbb{R}^{d}\times\mathbb{R}\}$ for $i=1,\ldots,n$, we want to learn a parameter $\xbf$ that minimizes 
\begin{align*}
\min_{\xbf\in\Xcal}\sum_{i=1}^n f_i(\xbf;\abf_i,b_i),
\end{align*}
where $\Xcal = [-1\,,\, 1]^{d}$ and $d=10$, i.e., $\xbf,\, \abf_i\in\Rset^{10}$. Here, $f_i$ are the loss functions defined over the dataset. For the purpose of our simulation, we will consider two loss functions, namely, quadratic loss and absolute loss functions. While the quadratic loss is strongly convex, the absolute loss is only convex.         

First, when $f_i$ are quadratic, we have the well-known least square problem given as
\begin{align*}
\min_{\xbf\in\Xcal}\;\sum_{i=1}^n(\abf_i^T\xbf-b_i)^2.
\end{align*}
Second, regression problems with absolute loss functions (or L$1$ norm) is often referred to as robust regression, which is known to be robust to outliers \cite{karst1958linear}, given as follows
\begin{align*}
\min_{\xbf\in\Xcal}\;\sum_{i=1}^n |\, \abf_i^T\xbf-b_i\,|.
\end{align*} 

We consider simulated training data sets, i.e., $(\abf_i,b_i)$ are generated randomly with uniform distribution between $[0,1]$.  We consider the performance of the distributed subgradient methods on an undirected connected graph of $50$ nodes, i.e., $\Gcal=(\Vcal,\Ecal)$ and n = $|\Vcal|=50$. Our graph is generated as follows. 
\begin{enumerate}
\item In each network, we first randomly generate the nodes' coordinates in the plane with uniform distribution.
\item Then any two nodes are connected if their distance is less than a reference number $r$, e.g, $r = 0.4$ for our simulations.
\item Finally we check whether the network is connected. If not we return to step $1$ and run the program again.
\end{enumerate}
To implement our algorithm, the adjacency matrix $\Abf$ is chosen as a lazy Metropolis matrix corresponding to $\Gcal$, i.e.,
\begin{align}
\Abf = [a_{ij}] = \left\{\begin{array}{ll}
\frac{1}{2(\max\{|\mathcal{N}_i| , |\Ncal_j|\})}, & \text{ if } (i,j) \in \Ecal\\
0, &\text{ if } (i,j)\notin\Ecal \text{ and } i\neq j\\
1-\sum_{j\in\Ncal_i}a_{ij},& \text{ if } i = j
\end{array}\right.\nonumber
\end{align}
It is straightforward to verify that the lazy Metropolis matrix $\Abf$ satisfies Assumption \ref{assump:doub_stoch}.

\subsection{Convergence of Function Values}

 \begin{figure} 
   \begin{subfigure}[b]{0.5\textwidth}
        \centering
    \includegraphics[width=\textwidth]{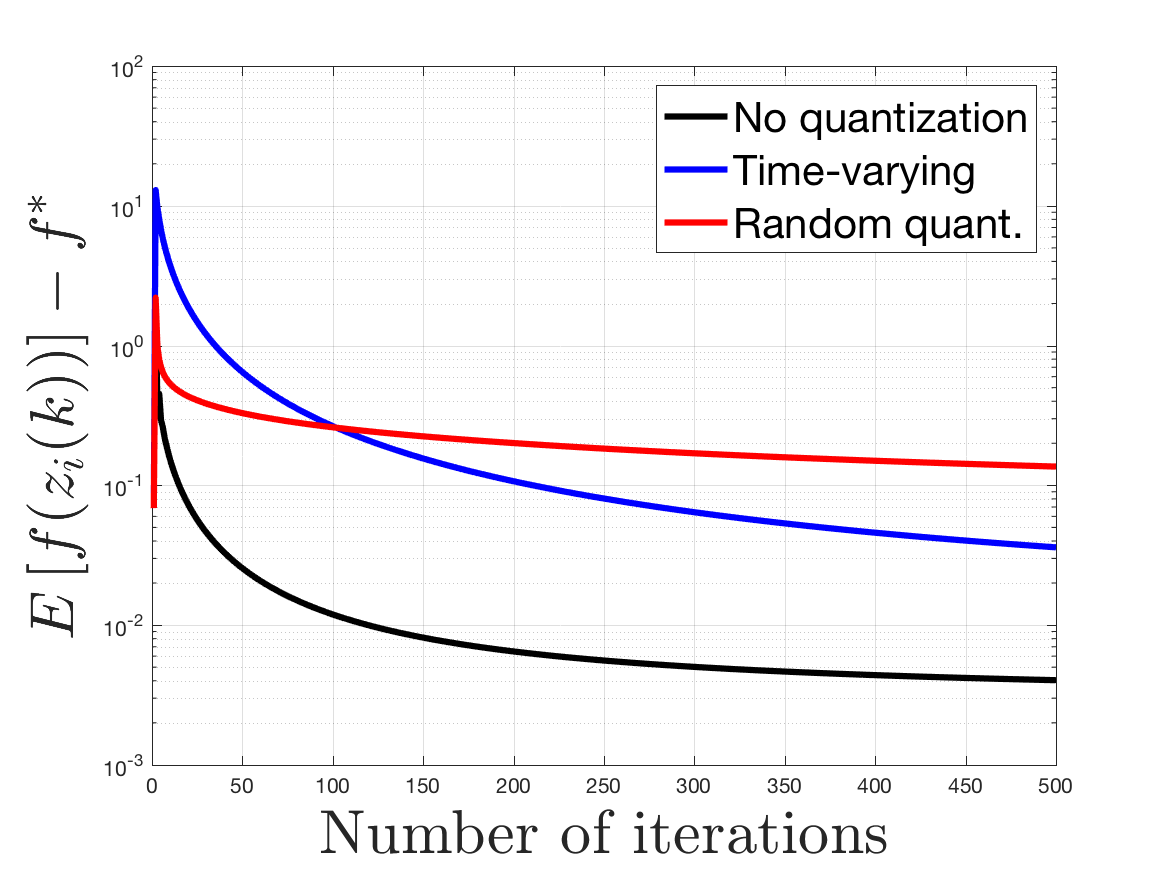}
    \caption{Quadratic loss functions}
     \label{fig:sc_func_value}
    \end{subfigure} 
   \begin{subfigure}[b]{0.5\textwidth}
        \centering \vspace{0.2cm}
    \includegraphics[width=\textwidth]{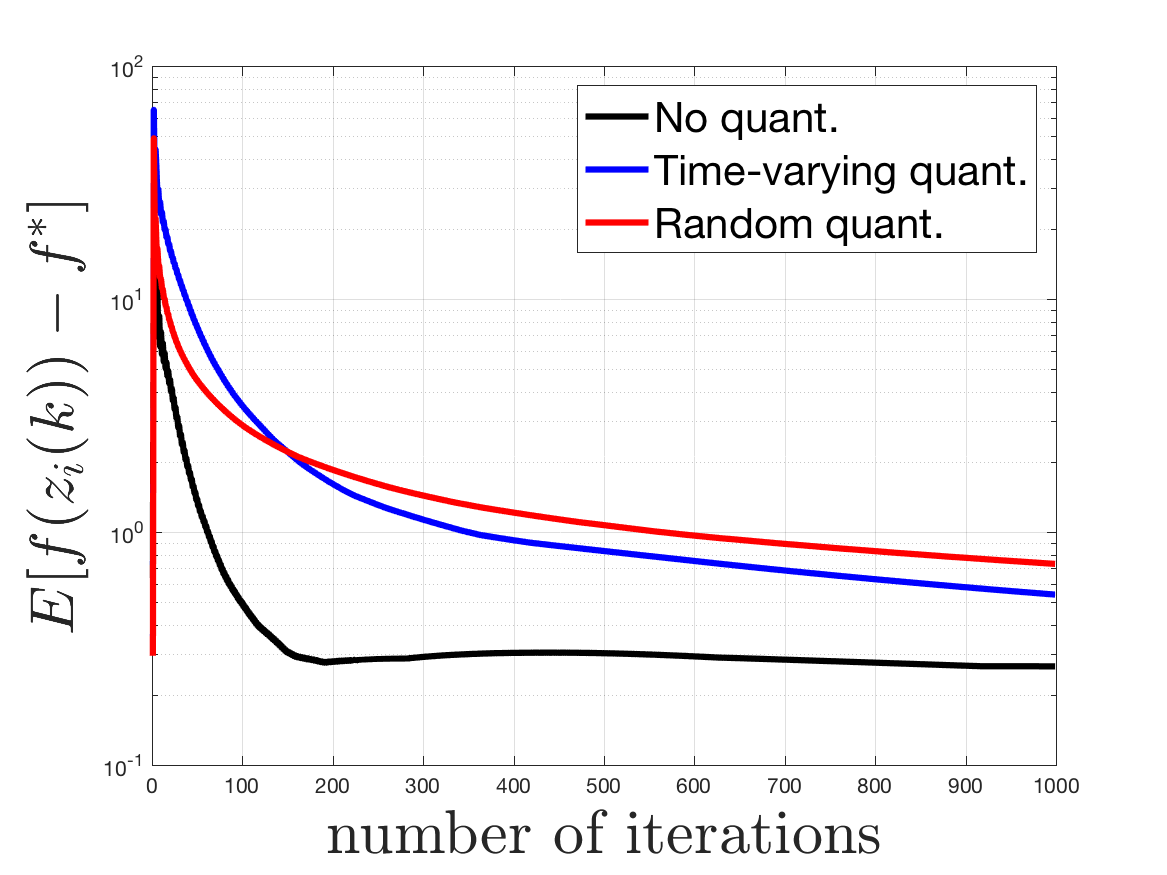}
        \caption{Absolute loss functions}
         \label{fig:c_func_value}
    \end{subfigure} 
 \caption{The convergence of function values using distributed subgradient methods without (\protect\blackline), with time-varying  (\protect\blueline), and with random  (\protect\redline) quantization for $n=50$ and $d=10$ are illustrated.}
 \label{fig:func_value}
 \end{figure}
In this simulation, we apply variants of distributed subgradient methods for solving the linear regression problems. In particular, we compare the performance of such methods for three different scenarios, namely, {\sf DSG} with no quantization (a.k.a Eq.\ \eqref{distributed:DSGD}), {\sf DSG} with time-varying quantization in \cite{DoanMR2018}, and the proposed stochastic variant of Eq.\ \eqref{distributed:DSGD} in Algorithm \ref{alg:QDSG}. The plots in Fig. \ref{fig:func_value} show the convergence of these three methods for both quadratic and absolute loss functions.  
 
Note that, to achieve an asymptotic convergence the work  in \cite{DoanMR2018} requires that the nodes eventually exchange an infinite number of bits. On the other hand, Algorithm \ref{alg:QDSG} in this paper assumes the nodes use a finite number of constant bits $b$ in their communication. However, as observed both in Fig. \ref{fig:sc_func_value} for quadratic loss and in Fig. \ref{fig:c_func_value} for absolute loss, Algorithm \ref{alg:QDSG} performs almost as well as the one in \cite{DoanMR2018}.

\subsection{Impacts of the Number of Bits $b$}
Here, we consider the impacts of the number of communication bits $b$ on the performance of Algorithm \ref{alg:QDSG}. In Fig. \ref{fig:b_impact} we plots the number of iterations, needed for the relative error $f(z_i(k))-f^*\,/\,f^*\leq 0.2$, as a function of $b$. As we can see the more bits we use the faster the algorithm converges. Moreover, the number of iterations required by the algorithm seems to be the same when $b$ is larger than $11$. This does make sense due to the numerical rounding of the computer program. Finally, the curves in Fig. \ref{fig:b_impact} seems to reflect the dependence of the rate of Algorithm \ref{alg:QDSG} on the variance $\Delta^2 = C/(2^b-1)$ within some constant factor $C$, which agrees with our results in Theorems \ref{thm:rate_conv} and \ref{thm_rate_conv_sc}.

 \begin{figure} 
   \begin{subfigure}[b]{0.5\textwidth}
        \centering
    \includegraphics[width=\textwidth]{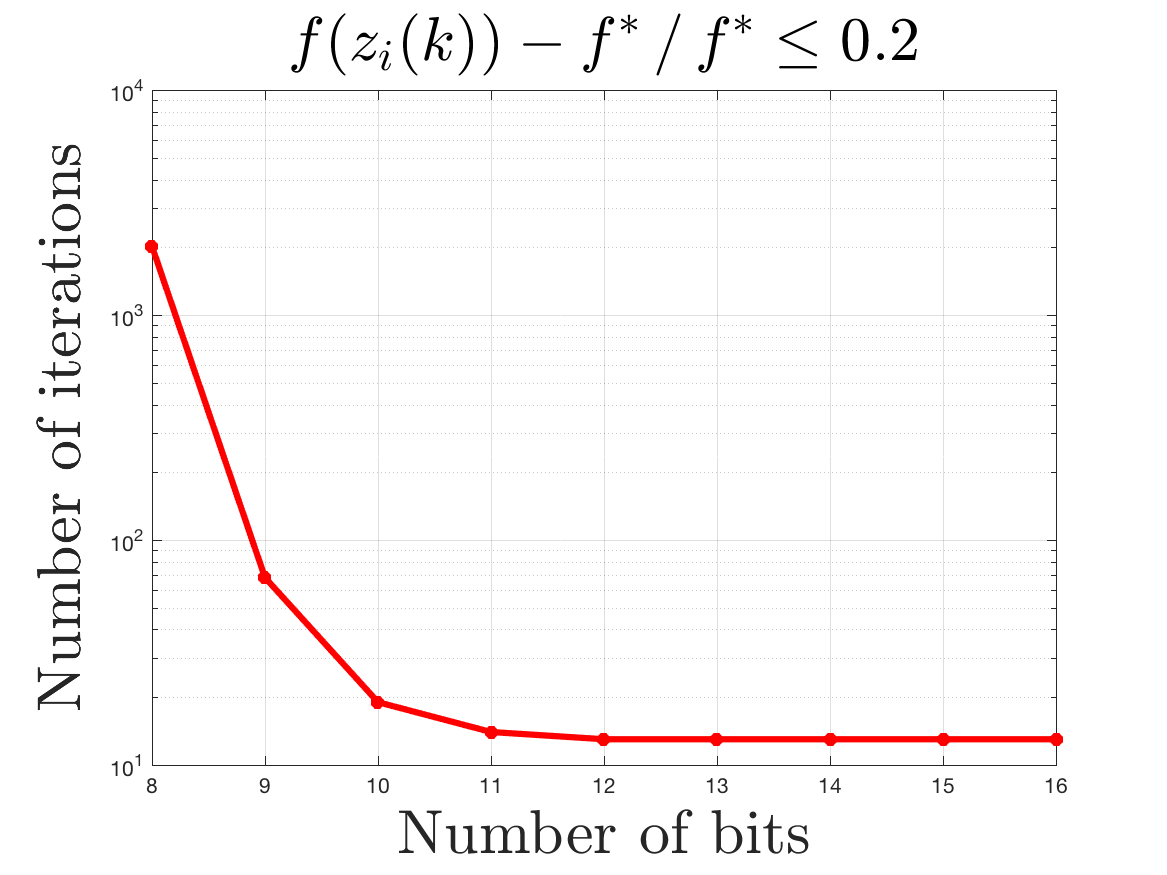}
    \caption{Quadratic loss functions}
    \end{subfigure} 
   \begin{subfigure}[b]{0.5\textwidth}
        \centering    \vspace{0.2cm}
    \includegraphics[width=\textwidth]{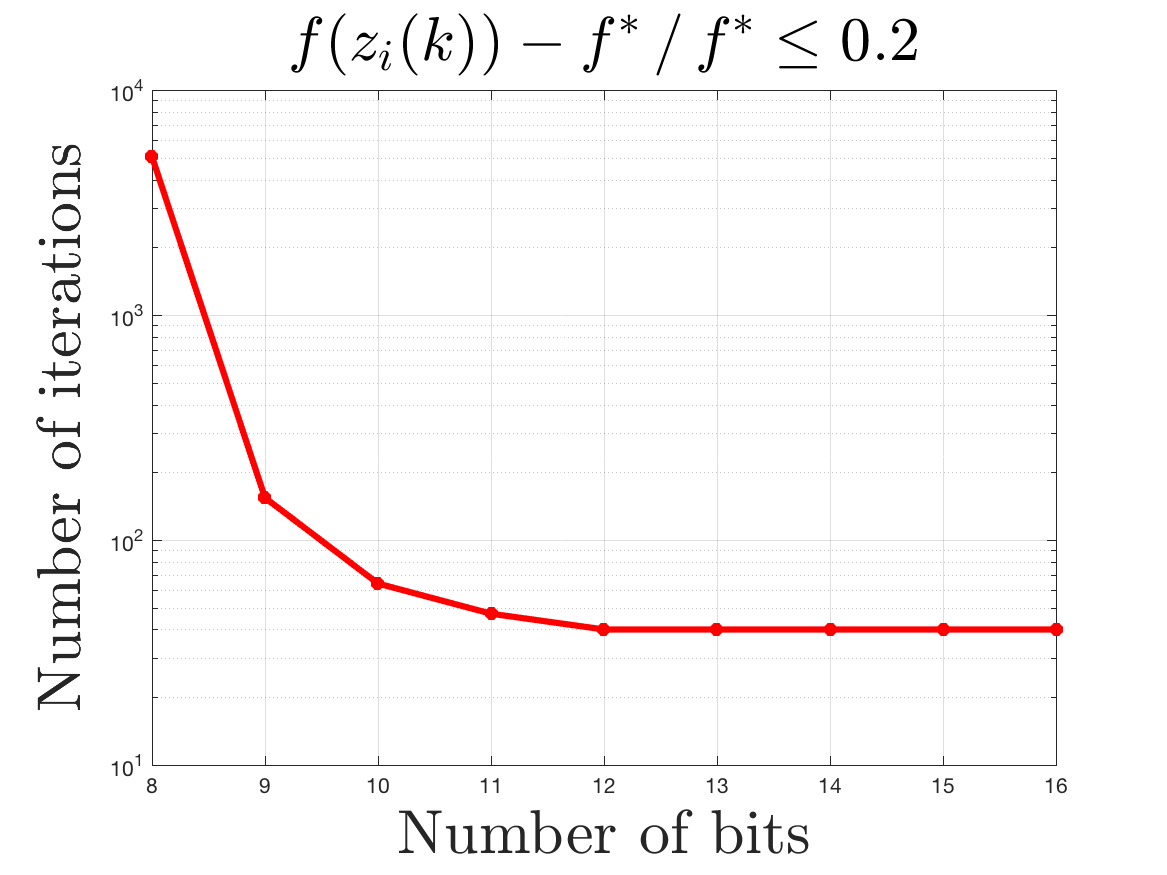}
        \caption{Absolute loss functions}
    \end{subfigure} 
 \caption{The number of iterations as a function of $b$ using distributed subgradient methods with random  quantization for $n=50$ and $d=10$ are illustrated.}
 \label{fig:b_impact}
 \end{figure}

%!TEX root = random_quantization.tex

\section{Proofs of Main Results}\label{sec:proof_main_results}
In this section, we present the proofs of our main results given in Section \ref{sec:analysis}. 
\subsection{Proof of Theorem \ref{thm:asymp_conv}}
By the convexity of $f_i$ we have
\begin{align*}
&- \frac{2\alpha(k)}{n}\sum_{i=1}^n \gbf_i^T(x_i(k))\Big(\xbf_i(k)-\xbf^*\Big) \notag\\
& \leq \frac{-2\alpha(k)}{n}\sum_{i=1}^n f_i(\xbf_i(k)) - f_i(\xbf^*)\notag\\
&= \frac{-2\alpha(k)}{n}\sum_{i=1}^nf_i(\xbf_i(k)) -f_i(\xbarbf(k))+f_i(\xbarbf(k)) - f_i(\xbf^*),
\end{align*}
which by the $L_i$-Lipschitz continuity of $f_i$ in Proposition \ref{prop:bounded_subg} yields 
\begin{align*}
&- \frac{2\alpha(k)}{n}\sum_{i=1}^n \gbf_i^T(x_i(k))\Big(\xbf_i(k)-\xbf^*\Big) \notag\\
&\leq \frac{2}{n}\sum_{i=1}^n\alpha(k)L_i\,\|\ybf_i(k)\,\|-\frac{2\alpha(k)}{n}\sum_{i=1}^nf_i(\xbarbf(k)) - f_i(\xbf^*)\notag\\
&\leq \frac{L^2\alpha^2(k)}{n\beta(k)} + \frac{\beta(k)\|\Ybf(k)\|^2}{n} -\frac{2\alpha(k)(f(\xbar(k))  - f^*)}{n}\cdot
\end{align*}
Substituting the preceding relation into Eq.\ \eqref{lem_opt_dist:Ineq} in Lemma \ref{lem:opt_dist} gives 
\begin{align}
&\Eset\left[\,r(k+1)\,|\,\Fcal_k\right]\notag\\ 
&\leq  r(k) + \frac{6L^2\alpha^2(k)}{n(1-\beta(0))}+\frac{2L^2\alpha^2(k)}{n\beta(k)} + \Delta^2\beta^2(k)\notag\\
&\qquad + \frac{2\beta(k)\|\Ybf(k)\|^2}{n}+ \frac{L^2\alpha^2(k)}{n\beta(k)} + \frac{\beta(k)\|\Ybf(k)\|^2}{n}\notag\\
&\qquad -\frac{2\alpha(k)(f(\xbar(k))  - f^*)}{n}\allowdisplaybreaks\notag\\
&= r(k) + \frac{6L^2\alpha^2(k)}{n(1-\beta(0))}+\frac{3L^2\alpha^2(k)}{n\beta(k)} + \Delta^2\beta^2(k)\notag\\
&\qquad +\frac{2L^2\alpha^2(k)}{n(1-\beta(k))} + \frac{3\beta(k)\|\Ybf(k)\|^2}{n}\notag\\
&\qquad -\frac{2\alpha(k)}{n}(f(\xbar(k))  - f^*).\label{thm_opt_conv:Eq1}
\end{align}
Since $\{\alpha(k),\beta(k)\}$ satisfy Eq.\ \eqref{lem_consensus:stepsizes}, Eq.\ \eqref{lem_consensus:FiniteSum} also holds, which implies that
{\small
\begin{align*}
\sum_{k=0}^{\infty}\left(\alpha^2(k) + \beta^2(k) + \frac{\alpha^2(k)}{\beta(k)} + \beta(k)\|\Ybf(k)\|^2 \right) < \infty. 
\end{align*}}
Thus, we can apply Lemma \ref{lem:Martingale} to Eq.\ \eqref{thm_opt_conv:Eq1} to have
\begin{align}
\begin{aligned}
&\lim_{k\rightarrow\infty}r(k)\qquad \text{ exits a.s. for each } \xbf^*\\ &\sum_{k=0}^\infty\alpha(k)\left(f(\xbarbf(k))  - f^*\right) < \infty\qquad \text{a.s.}
\end{aligned}\label{thm_opt_conv:Eq2}
\end{align}
Consequently, since $\sum_{k=0}^\infty\alpha(k)=\infty$, Eq.\ \eqref{thm_opt_conv:Eq2}  implies
\begin{align*}
\liminf_{k\rightarrow\infty} f(\xbarbf(k)) = f^* \text{ a.s.}
\end{align*}
Let $\{\xbarbf(k_{\ell})\}$ be a subsequence of $\{\xbarbf(k)\}$ such that 
\begin{align*}
\lim_{\ell\rightarrow\infty} f(\xbarbf(k_{\ell})) = \liminf_{k\rightarrow\infty} f(\xbarbf(k)) = f^* \text{ a.s.}
\end{align*}
Since $\{\|\,\xbarbf(k)-x^*\,\|\}$ converges, the subsequence $\{\xbarbf(k_{\ell})\}$ is bounded. Hence, there is a convergent subsequence of $\{\xbarbf(k_{\ell})\}$, which converges to some minimizer $\widetilde{\xbf}$ of problem \eqref{prob:obj} a.s. since $\lim_{\ell\rightarrow\infty} f(\xbarbf(k_{\ell})) = f^*$ a.s. In addition, since $\Big\{\|\,\xbarbf(k)-\xbf^*\,\|\Big\} \text{ converges a.s. for each } x^*$, and in particular, $\Big\{\|\,\xbarbf(k)-\widetilde{\xbf}\,\|\Big\} \text{ converges a.s.}$,
we obtain
\begin{align*}
\lim_{k\rightarrow\infty}\xbarbf(k) = \widetilde{\xbf}\text{ a.s. },
\end{align*}
which together with Eq.\ \eqref{lem_consensus:AsympCov} implies Eq.\ \eqref{thm_opt:asympconv}.

\subsection{Proof of Theorem \ref{thm:rate_conv}}
Taking the expectation on both sides of Eq.\ \eqref{thm_opt_conv:Eq1} yields 
\begin{align}
& \Eset\left[\,r(k+1)\,\right]\notag\\ 
&\leq \Eset\left[\,r(k)\,\right]+ \frac{6L^2\alpha^2(k)}{n(1-\beta(0))}+\frac{3L^2\alpha^2(k)}{n\beta(k)} + \Delta^2\beta^2(k)\notag\\
&\qquad+\frac{3\beta(k)\Eset[\,\|\Ybf(k)\|^2\,]}{n} - \frac{2\alpha(k)}{n}\Eset[\,f(\xbarbf(k))  - f^*\,].\label{thm_rate_conv:Eq1}
\end{align}
Fix some $\ell\in\Vcal$ and consider
\begin{align*}
&-\frac{2\alpha(k)}{n}\Big(f(\xbarbf(k))  - f^*\Big)\notag\\ 
& = -\frac{2\alpha(k)}{n}\Big(f(\xbf_\ell(k))- f^*\Big)\nonumber\\
&\qquad -\frac{2\alpha(k)}{n}\sum_{i=1}^nf_i(\xbarbf(k))-f_i(\xbf_\ell(k)),
\end{align*}
which by the $L_i$-Lipschitz continuity of $f_i$ yields
\begin{align*}
&-\frac{2\alpha(k)}{n}\Big(f(\xbarbf(k))  - f^*\Big)\notag\\ 
&\leq -\frac{2\alpha(k)}{n}\Big(f(\xbf_\ell(k))- f^*\Big)\nonumber\\
&\qquad + \frac{2}{n}\sum_{i=1}^n\alpha(k)L_i\,\|\,\xbarbf(k)-\xbf_\ell(k)\,\|\notag\\
&\leq -\frac{2\alpha(k)}{n}\Big(f(\xbf_\ell(k))- f^*\Big)\nonumber\\
&\qquad + \frac{1}{n}\sum_{i=1}^n\left(\frac{L_i^2\alpha^2(k)}{\beta(k)}+\beta(k)\|\xbarbf(k)-\xbf_\ell(k)\|^2\right)\notag\\
&\leq  -\frac{2\alpha(k)}{n}\Big(f(\xbf_\ell(k))- f^*\Big)+ \frac{L^2\alpha^2(k)}{n\beta(k)} + \beta(k)\|\Ybf(k)\|^2,
\end{align*}
where the second inequality is due to the Cauchy-Schwarz inequality. Substituting the preceding relation into Eq.\ \eqref{thm_opt_conv:Eq1} gives 
\begin{align*}
& \Eset\left[\,r(k+1)\,\right]\notag\\ 
&\leq  \Eset\left[\,r(k)\,\right] + \frac{6L^2\alpha^2(k)}{n(1-\beta(0))}+\frac{4L^2\alpha^2(k)}{n\beta(k)} + \Delta^2\beta^2(k)\notag\\
&\qquad +4\beta(k)\,]\Eset[\,\|\Ybf(k)\|^2- \frac{2\alpha(k)}{n}\Eset[\,f(\xbf_{\ell}(k))  - f^*\,].
\end{align*}
Summing up both sides of the preceding relation over $k = 0,\ldots,K$ for some $K\geq 0$ and rearranging we obtain
\begin{align*}
&\sum_{k=0}^{K}\alpha(k)\Eset[\,f(x_\ell(k))- f^*\,]\notag\\
&\leq \frac{n\Eset\left[\,r(0)\,\right]}{2} + \frac{3L^2}{1-\beta(0)}\sum_{k=0}^{K}\alpha^2(k) + 2L^2\sum_{k=0}^{K}\frac{\alpha^2(k)}{\beta(k)}\notag\\
&\qquad + \frac{n\Delta^2}{2}\sum_{k=0}^{K}\beta^2(k)+ 2n\sum_{k=0}^{K}\beta(k)\Eset[\,\|\Ybf(k)\|^2\,],
\end{align*}
which by applying Eq. \eqref{lem_consensus:UpperBound_rate} to upper bound the last term on the right-hand side gives
\begin{align}
&\sum_{k=0}^{K}\alpha(k)\Eset[\,f(x_\ell(k))- f^*\,]\notag\\
&\leq \frac{n\Eset\left[\,r(0)\,\right]}{2} + \frac{3L^2}{1-\beta(0)}\sum_{k=0}^{K}\alpha^2(k) + 2L^2\sum_{k=0}^{K}\frac{\alpha^2(k)}{\beta(k)}\notag\\
&\qquad + \frac{n\Delta^2}{2}\sum_{k=0}^{K}\beta^2(k)+\frac{2n\Eset[\,\|\Ybf(0)\|^2\,]}{1-\sigma_2}\notag\\
&\qquad +\frac{2n^2\sigma_2^2\Delta^2}{1-\sigma_2} \sum_{k=0}^{K}\beta^2(k) + \frac{16nL^2}{(1-\sigma_2)^2}\sum_{k=0}^{K}\frac{\alpha^2(k)}{\beta(k)}\allowdisplaybreaks\notag\\
&\leq \frac{n\Eset\left[\,r(0)\,\right]}{2}  +\frac{2n\Eset[\,\|\Ybf(0)\|^2\,]}{1-\sigma_2}+ \frac{3L^2}{1-\beta(0)}\sum_{k=0}^{K}\alpha^2(k) \notag\\
&\qquad + \frac{18nL^2}{(1-\sigma_2)^2}\sum_{k=0}^{K}\frac{\alpha^2(k)}{\beta(k)} +\frac{5n^2\Delta^2}{2(1-\sigma_2)} \sum_{k=0}^{K}\beta^2(k)\cdot \label{thm_rate_conv:Eq2}
\end{align}
Recall from Eq.\ \eqref{thm_rate_conv:stepsizes} that
\begin{align*}
\alpha(k) = \frac{1}{(k+2)^{3/4}},\qquad \beta(k) = \frac{1}{(k+2)^{1/2}},
\end{align*}
implying that $1\,/\,(1-\beta(0))\leq 4$. In addition, using the integral test we have
\begin{align*}
& \sum_{k=0}^{K}\alpha(k) = \sum_{k=0}^{K}\frac{1}{(k+2)^{3/4}}\geq  \int_{0}^{K}\frac{ds}{(s+2)^{3/4}}\geq 4(K+2)^{1/4}\\
&\sum_{k=0}^{K}\alpha^{2}(k) \leq 1 + \int_{0}^{K}\frac{ds}{(s+2)^{3/2}}\leq 5\\
&\sum_{k=0}^{K}\beta^{2}(k) \leq 1 + \int_{0}^{K}\frac{ds}{s+2}\leq 1 + \ln(K+2)\\
&\sum_{k=0}^{K}\frac{\alpha^2(k)}{\beta(k)}\leq 1+\int_{0}^{K}\frac{ds}{s+2}\leq 1 + \ln(K+2)\cdot 
\end{align*}
Thus, dividing both sides of Eq.\ \eqref{thm_rate_conv:Eq2} by $\sum_{k=0}^{K}\alpha(k)$ and by the Jensen's inequality yields  Eq.\ \eqref{thm_rate_conv:Ineq} , i.e.,
\begin{align*}
&\Eset\left[f\left(\frac{\sum_{k=0}^{K}\alpha(k)x_{\ell}(k)}{\sum_{k=0}^{K}\alpha(k)}\right)\right] - f^*\notag\\
&\leq \frac{\sum_{k=0}^{K}\alpha(k)\Eset[\,f(x_\ell(k))\,]}{\sum_{k=0}^{K}\alpha(k)}- f^*\notag\\
&\leq \frac{n\Eset\left[\,r(0)\,\right]}{8(K+2)^{1/4}}  +\frac{n\Eset[\,\|\Ybf(0)\|^2\,]}{2(1-\sigma_2)(K+2)^{1/4}}+ \frac{16L^2}{(K+2)^{1/4}}\notag\\
&\qquad +\frac{9nL^2(1+\ln(K+2))}{2(1-\sigma)^2(K+2)^{1/4}}+\frac{5n^2\Delta^2(1+\ln(K+2))}{8(K+2)^{1/4}}\cdot
\end{align*}

\subsection{Proof of Theorem \ref{thm_rate_conv_sc}}
Fix some $\ell\in\Vcal$. First, the strong convexity of $f_i$ gives 
\begin{align*}
&-\frac{2\alpha(k)}{n}\sum_{i=1}^n \gbf_i(\xbf_i(k))(\xbf_i(k)-x^*) \notag\\
& \leq \frac{-2\alpha(k)}{n}\sum_{i=1}^n \Big(f_i(\xbf_i(k)) - f_i(\xbf^*)+\frac{\mu_i}{2}\|\xbf_i(k)-x^*\|^2\Big)\notag\\
&= \frac{-2\alpha(k)}{n}\sum_{i=1}^nf_i(\xbf_i(k)) -f_i(\xbarbf(k))+f_i(\xbarbf(k)) - f_i(\xbf^*)\notag\\
&\qquad -\frac{\alpha(k)}{n}\sum_{i=1}^n \mu_i\|\xbf_i(k)-\xbf^*\|^2,
\end{align*}
which by the $L_i$-Lipschitz continuity of $f_i$ yields
\begin{align}
&-\frac{2\alpha(k)}{n}\sum_{i=1}^n \gbf_i(\xbf_i(k))(\xbf_i(k)-x^*) \notag\\
&\leq \frac{2}{n}\sum_{i=1}^n\alpha(k)L_i\,\|\ybf_i(k)\,\|-\frac{2\alpha(k)}{n}\sum_{i=1}^n f_i(\xbarbf(k)) - f_i(\xbf^*)\notag\\
&\qquad -\frac{\mu\alpha(k)}{n}\sum_{i=1}^n\|\xbf_i(k)-\xbf^*\|^2\notag\\
&\leq \frac{L^2\alpha^2(k)}{n\beta(k)} + \frac{\beta(k)\|\Ybf(k)\|^2}{n} -\mu\alpha(k)r(k)\notag\\
&\qquad -\frac{2\alpha(k)}{n}\sum_{i=1}^n f_i(\xbarbf(k)) -f_i(\xbf_{\ell}(k)) +f_i(\xbf_{\ell}(k))- f_i(\xbf^*)\allowdisplaybreaks\notag\\
&\leq \frac{2L^2\alpha^2(k)}{n\beta(k)} + 2\beta(k)\|\Ybf(k)\|^2 -\mu\alpha(k)r(k)\notag\\
&\qquad -\frac{2\alpha(k)}{n}\Big(f(\xbf_{\ell}(k)) - f(\xbf^*)\Big),\label{thm_rate_conv_sc:Eq1}
\end{align}
where in the second inequality we use the Cauchy-Schwartz inequality to have
\begin{align*}
-\sum_{i=1}^n\frac{\|\xbf_i(k)-\xbf^*\|^2}{n} \leq -\left\|\sum_{i=1}^n\frac{1}{n}(\xbf_i(k)-\xbf^*)\right\|^2 = -r(k)      
\end{align*}
Thus, substituting Eq.\ \eqref{thm_rate_conv_sc:Eq1} into Eq.\ \eqref{lem_opt_dist:Ineq} we obtain
\begin{align}
& \Eset\left[\,r(k+1)\,|\,\Fcal_k\,\right]\notag\\ 
&\leq r(k) + \frac{6L^2\alpha^2(k)}{n(1-\beta(0))}+\frac{2L^2\alpha^2(k)}{n\beta(k)} + \Delta^2\beta^2(k)\notag\\
&\qquad  + \frac{2\beta(k)\|\Ybf(k)\|^2}{n}+ \frac{2L^2\alpha^2(k)}{n\beta(k)} + 2\beta(k)\|\Ybf(k)\|^2\notag\\
&\qquad  -\mu\alpha(k)r(k)-\frac{2\alpha(k)}{n}\Big(f(\xbf_{\ell}(k)) - f(\xbf^*)\Big)\notag\\
&= (1-\mu\alpha(k))r(k) + \frac{6L^2\alpha^2(k)}{n(1-\beta(0))}+\frac{4L^2\alpha^2(k)}{n\beta(k)} + \Delta^2\beta^2(k)\notag\\
&\qquad + 4\beta(k)\|\Ybf(k)\|^2 -\frac{2\alpha(k)}{n}\Big(f(\xbf_{\ell}(k)) - f(\xbf^*)\Big).\label{thm_rate_conv_sc:Eq2}
\end{align}
Note that since $\alpha(k)$ satisfies Eq.\ \eqref{thm_rate_conv_sc:stepsizes}, we have $$1-\mu\alpha(k)\leq 1-\frac{1}{k+2}\cdot$$ Then, using Eq.\ \eqref{thm_rate_conv_sc:stepsizes} into Eq. \eqref{thm_rate_conv_sc:Eq2} gives
\begin{align*}
& \Eset\left[\,r(k+1)\,|\,\Fcal_k\,\right]\notag\\ 
&\leq \left(1-\frac{1}{k+2}\right)r(k) + \frac{6L^2\alpha^2(k)}{n(1-\beta(0))}+\frac{4L^2\alpha^2(k)}{n\beta(k)} + \Delta^2\beta^2(k)\notag\\
&\qquad + 4\beta(k)\|\Ybf(k)\|^2-\frac{2\alpha(k)}{n}\Big(f(\xbf_{\ell}(k)) - f(\xbf^*)\Big)\notag\\
&\leq \frac{k+1}{k+2}r(k)  + \frac{6L^2\alpha^2(k)}{n(1-\beta(0))}+\frac{4L^2\alpha^2(k)}{n\beta(k)} + \Delta^2\beta^2(k)\notag\\
&\qquad + 4\beta(k)\|\Ybf(k)\|^2 -\frac{2\alpha(k)}{n}\Big(f(\xbf_{\ell}(k)) - f(\xbf^*)\Big)\allowdisplaybreaks\nonumber\\
&\leq \frac{k+1}{k+2}r(k) + \frac{6L^2\alpha^2(0)}{n(1-\beta(0))(k+2)^2} +  \frac{4L^2\alpha^2(0)}{n\beta(0)(k+2)^{4/3}}\notag\\
&\qquad +\frac{\Delta^2\beta^2(0)}{(k+2)^{4/3}} + 4\beta(k)\|\Ybf(k)\|^2 \notag\\ 
&\qquad -\frac{2\alpha(0)}{n(k+2)}\Big(f(\xbf_{\ell}(k)) - f(\xbf^*)\Big),
\end{align*}
which when multiplying both sides by $(k+2)$ yields 
\begin{align*}
& (k+2)\Eset\left[\,r(k+1)\,|\,\Fcal_k\,\right]\notag\\
&\leq (k+1)r(k) + \frac{6L^2\alpha^2(0)}{n(1-\beta(0))(k+2)} +  \frac{4L^2\alpha^2(0)}{n\beta(0)(k+2)^{1/3}}\notag\\
&\qquad +\frac{\Delta^2\beta^2(0)}{(k+2)^{1/3}} + 4\beta(k)\|\Ybf(k)\|^2 \notag\\ 
&\qquad -\frac{2\alpha(0)}{n}\Big(f(\xbf_{\ell}(k)) - f(\xbf^*)\Big).
\end{align*}
By iteratively updating over $k$ of the preceding relation we have 
\begin{align*}
& (k+2)\Eset\left[\,r(k+1)\,|\,\Fcal_k\,\right]\notag\\ 
&\leq r(0) +  \sum_{t=0}^{k}\frac{6L^2\alpha^2(0)}{n(1-\beta(0))(t+2)^2}(t+2)\notag\\ 
&\qquad +  \sum_{t=0}^{k}\left(\frac{\Delta^2\beta^2(0)}{(t+2)^{4/3}} +  \frac{4L^2\alpha^2(0)}{n\beta(0)(t+2)^{4/3}}\right)(t+2)\notag\\
&\qquad + 4\sum_{t=0}^{k}(t+2)\beta(t)\|\Ybf(t)\|^2 - \sum_{t=0}^{k}\frac{2\alpha(0)}{n}\Big(f(\xbf_{\ell}(t)) - f(\xbf^*)\Big)\notag\\
&\leq r(0)+  \frac{6L^2\alpha^2(0)}{n(1-\beta(0))}\sum_{t=0}^{k}\frac{1}{(t+2)}\notag\\
&\qquad +  \sum_{t=0}^{k}\left(\frac{\Delta^2\beta^2(0)}{(t+2)^{1/3}} +  \frac{4L^2\alpha^2(0)}{n\beta(0)(t+2)^{1/3}}\right)\notag\\
&\qquad + 4\sum_{t=0}^{k}(t+2)\beta(t)\|\Ybf(t)\|^2 - \frac{2\alpha(0)}{n}\sum_{t=0}^{k}\Big(f(\xbf_{\ell}(t)) - f(\xbf^*)\Big).
\end{align*}
Taking the expectation of the preceding inequality and rearranging gives 
\begin{align*}
& \frac{2\alpha(0)}{n}\sum_{t=0}^{k}\Eset\Big[f(\xbf_{\ell}(t)) - f(\xbf^*)\Big] +  (k+2)\Eset\left[\,r(k+1)\,\right]\notag\\ 
&\leq \Eset[\,r(0)\,]+  \frac{6L^2\alpha^2(0)}{n(1-\beta(0))}\sum_{t=0}^{k}\frac{1}{(t+2)}\notag\\
&\qquad +  \sum_{t=0}^{k}\left(\frac{\Delta^2\beta^2(0)}{(t+2)^{1/3}} +  \frac{4L^2\alpha^2(0)}{n\beta(0)(t+2)^{1/3}}\right)\notag\\
&\qquad + 4\sum_{t=0}^{k}(t+2)\beta(t)\Eset[\,\|\Ybf(t)\|^2\,],
\end{align*}
which when dropping the nonnegative term $\Eset\left[\,r(k+1)\,\right]$ and dividing both sides by $(k+2)\,/\,n$ yields 
\begin{align}
&\frac{2\alpha(0)}{(k+2)}\sum_{t=0}^{k}\Eset \Big[f(\xbf_{\ell}(t)) - f(\xbf^*)\Big]\notag\\
&\leq \frac{n\Eset[\,r(0)\,]}{k+2}+\frac{6L^2\alpha^2(0)}{(1-\beta(0))(k+2)}\sum_{t=0}^{k}\frac{1}{(t+2)}\notag\\ 
&\qquad +  \frac{n}{k+2}\sum_{t=0}^{k}\left(\frac{\Delta^2\beta^2(0)}{(t+2)^{1/3}} +  \frac{4L^2\alpha^2(0)}{n\beta(0)(t+2)^{1/3}}\right)\notag\\
&\qquad + \frac{4n}{(k+2)}\sum_{t=0}^{k}(t+2)\beta(t)\Eset[\,\|\Ybf(t)\|^2\,].\label{thm_rate_conv_sc:Eq3}
\end{align}
We now analyze the last term on the right-hand side of Eq.\ \eqref{thm_rate_conv_sc:Eq3}. First, Eq.\ \eqref{lem_consensus:UpperBound} yields
\begin{align}
&\sum_{t=0}^{k}(t+2)\beta(t)\Eset[\,\|\Ybf(t)\|^2\,]\notag\\
&\leq \sum_{t=0}^{k}\frac{t+2}{1-\sigma_2}\Big[\Eset[\,\|\Ybf(t)\|^2\,]-\Eset[\,\|\Ybf(t+1)\|^2\,]\Big]\notag\\
&\qquad + \sum_{t=0}^{k}(t+2)\left(\frac{n\sigma_2^2\Delta^2\beta^2(t)}{1-\sigma_2} + \frac{8L^2\alpha^2(t)}{(1-\sigma_2)^2\beta(t)}\right)\notag\\
&\leq \frac{\sum_{t=0}^{k}\Eset[\,\|\Ybf(t)\|^2\,]}{1-\sigma_2}+ \sum_{t=0}^{k}\frac{n\sigma_2^2\Delta^2\beta^2(0)}{(1-\sigma_2)(t+2)^{1/3}}\notag\\
&\qquad + \sum_{t=0}^{k}\frac{8L^2\alpha^2(0)}{\beta(0)(1-\sigma_2)^2(t+2)^{1/3}}\cdot\label{thm_rate_conv_sc:Eq3a}
\end{align}
Recall that, by Eq.\ \eqref{thm_rate_conv_sc:stepsizes} we have
\begin{align*}
1-(1-\sigma_2)\beta(k) = 1-\frac{(1-\sigma_2)b}{(k+2)^{2/3}}\leq 1-\frac{1}{k+2}=\frac{k+1}{k+2}\cdot
\end{align*}
Second, using Eq.\ \eqref{lem_consensus:UpperBound} one more time gives
\begin{align*}
&\Eset[\,\|\Ybf(k+1)\,\|^2]\notag\\ 
&\leq (1-(1-\sigma_2)\beta(k))\Eset[\,\|\Ybf(k)\,\|^2] \notag\\
&\quad+ n\sigma_2^2\Delta^2\beta^2(k)+\frac{8L^2\alpha^2(k)}{(1-\sigma_2)\beta(k)}\notag\\
&= \left(\frac{k+1}{k+2}\right)\Eset[\,\|\Ybf(k)\,\|^2]+ n\sigma_2^2\Delta^2\beta^2(k)+\frac{8L^2\alpha^2(k)}{(1-\sigma_2)\beta(k)},
\end{align*}
which implies that
\begin{align*}
&(1-\sigma_2)\Eset[\,\|\Ybf(k+1)\,\|^2]\notag\\
&\leq \frac{1}{k+2}\frac{\Eset[\,\|\Ybf(0)\|\,]}{1-\sigma_2}+ \sum_{t=0}^{k} \frac{n\sigma_2^2\Delta^2\beta^2(t)}{(1-\sigma_2)}\prod_{u=t+1}^{k}\frac{u+1}{u+2}\notag\\ 
&\qquad +  \sum_{t=0}^{k} \frac{8L^2\alpha^2(t)}{(1-\sigma_2)^2\beta(t)}\prod_{u=t+1}^{k}\frac{u+1}{u+2}\notag\\
&=\frac{1}{k+2}\frac{\Eset[\,\|\Ybf(0)\|\,]}{1-\sigma_2} + \frac{1}{k+2}\sum_{t=0}^{k} \frac{n\sigma_2^2\Delta^2\beta^2(0)}{(1-\sigma_2)(t+2)^{1/3}}\\
&\qquad +\frac{1}{k+2}\sum_{t=0}^{k} \frac{8L^2\alpha^2(0)}{\beta(0)(1-\sigma_2)^2(t+2)^{1/3}}\allowdisplaybreaks\\
&\leq \frac{1}{k+2}\frac{\Eset[\,\|\Ybf(0)\|\,]}{1-\sigma_2} + \frac{3n\sigma_2^2\Delta^2\beta^2(0)}{2(1-\sigma_2)(k+2)^{1/3}}\\
&\qquad +\frac{12L^2\alpha^2(0)}{\beta(0)(1-\sigma_2)^2(k+2)^{1/3}},
\end{align*}
where the last inequality we use the integral test to have 
\begin{align}
\sum_{t=0}^{k}\frac{1}{(t+2)^{1/3}}\leq \frac{3}{2}(k+2)^{2/3}.\label{thm_rate_conv_sc:Eq3b}     
\end{align}
Thus, using the equation above into Eq.\ \eqref{thm_rate_conv_sc:Eq3a} we have
\begin{align*}
&\sum_{t=-1}^{k}(t+2)\beta(t)\Eset[\,\|\Ybf(t)\|^2\,]\notag\\
&\leq \frac{\Eset[\,\|\Ybf(0)\|\,](1+\ln(k+2))}{1-\sigma_2} + \sum_{t=0}^{k}\frac{3n\sigma_2^2\Delta^2\beta^2(0)}{(1-\sigma_2)^2(k+2)^{1/3}}\\
&\qquad +\sum_{t=0}^{k}\frac{20L^2\alpha^2(0)}{\beta(0)(1-\sigma_2)^3(k+2)^{1/3}},
\end{align*}
which when substituting into Eq.\ \eqref{thm_rate_conv_sc:Eq3} yields
\begin{align}
&\frac{2\alpha(0)}{(k+2)}\sum_{t=0}^{k}\Eset \Big[f(\xbf_{\ell}(t)) - f(\xbf^*)\Big]\notag\\
&\leq \frac{n\Eset[\,r(0)\,]}{(k+2)} + \frac{4n\Eset[\,\|\Ybf(0)\|\,](1+\ln(k+2))}{(1-\sigma_2)(k+2)} \notag\\ 
&\qquad + \frac{6L^2\alpha^2(0)}{(1-\beta(0))(k+2)}\sum_{t=0}^{k}\frac{1}{(t+2)}\notag\\
&\qquad +\frac{n}{k+2}\sum_{t=0}^{k}\frac{\Delta^2\beta^2(0)}{(t+2)^{1/3}}\displaybreak[0]\notag\\ 
&\qquad +  \frac{1}{k+2}\sum_{t=0}^{k}  \frac{4L^2\alpha^2(0)}{\beta(0)(t+2)^{1/3}}\notag\\
&\qquad + \frac{12n^2\beta^2(0)\sigma_2^2\Delta^2}{(1-\sigma_2)^2(k+2)}\sum_{t=0}^{k}\frac{1}{(t+2)^{1/3}}\nonumber\\
&\qquad+\frac{40nL^2\alpha(0)}{\beta(0)(1-\sigma_2)^3(k+2)}\sum_{t=0}^{k}\frac{1}{(t+2)^{1/3}}\cdot\label{thm_rate_conv_sc:Eq4}
\end{align}
Using Eq.\ \eqref{thm_rate_conv_sc:Eq3b} into Eq. \eqref{thm_rate_conv_sc:Eq4} gives
\begin{align}
&2\alpha(0)\left[\frac{\sum_{t=0}^{k}\Eset \Big[f(\xbf_{\ell}(t))\Big]}{k+1} - f(\xbf^*)\right]\notag\\
&\leq \frac{2\alpha(0)(k+1)}{k+2}\left[\frac{\sum_{t=0}^{k}\Eset \Big[f(\xbf_{\ell}(t))\Big]}{k+1} - f(\xbf^*)\right]\notag\\
&\leq \frac{n\Eset[\,r(0)\,]}{(k+2)} + \frac{4n\Eset[\,\|\Ybf(0)\|\,](1+\ln(k+2))}{(1-\sigma_2)(k+2)}\notag\\ 
&\qquad + \frac{6L^2\alpha^2(0)(1+\ln(k+2))}{(1-\beta(0))(k+2)}\notag\\
&\qquad +\frac{n\Delta^2\beta^2(0)}{(k+2)^{1/3}} +  \frac{4L^2\alpha^2(0)}{\beta(0)(k+2)^{1/3}} \notag\\
&\qquad + \frac{8n^2\sigma_2^2\Delta^2\beta^2(0)}{(1-\sigma_2)^2(k+2)^{1/3}} + \frac{27nL^2\alpha(0)}{\beta(0)(1-\sigma_2)^3(k+2)^{1/3}}\cdot\label{thm_rate_conv_sc:Eq4a}
\end{align}
The Jsensen's inequality and the strong convexity of $f$ give 
\begin{align}
\frac{\mu}{2}\Eset[\,\|\zbf_{\ell}(k)-\xbf^*\|^2\,]&\leq \Eset[\,f(\zbf_{\ell}(k))\,] - f^*\notag\\
&\leq \frac{\sum_{t=0}^{k}\Eset \Big[f(\xbf_{\ell}(t))\Big]}{k+1} - f(\xbf^*). \label{thm_rate_conv_sc:Eq4b}
\end{align}
Moreover, note that $\mu\alpha(0) > 1$. Thus, using \eqref{thm_rate_conv_sc:Eq4a} and \eqref{thm_rate_conv_sc:Eq4b} gives Eq.\ \eqref{thm_rate_conv_sc:Ineq}, i.e. 
\begin{align*}
&\Eset[\,\|\zbf_{\ell}(k)-\xbf^*\|^2\,]\\
&\leq \frac{n\Eset[\,r(0)\,]}{(k+2)} + \frac{4n\Eset[\,\|\Ybf(0)\|\,](1+\ln(k+2))}{(1-\sigma_2)(k+2)}\notag\\ 
&\qquad + \frac{6L^2\alpha^2(0)(1+\ln(k+2))}{(1-\beta(0))(k+2)}\notag\\
&\qquad +\frac{n\Delta^2\beta^2(0)}{(k+2)^{1/3}} +  \frac{4L^2\alpha^2(0)}{\beta(0)(k+2)^{1/3}} \notag\\
&\qquad + \frac{8n^2\sigma_2^2\Delta^2\beta^2(0)}{(1-\sigma_2)^2(k+2)^{1/3}} + \frac{27nL^2\alpha(0)}{\beta(0)(1-\sigma_2)^3(k+2)^{1/3}}\cdot
\end{align*}

\section{Concluding Remarks}\label{sec:conclusion}
In this paper, we consider distributed optimization over networks of nodes under quantized communication. For solving such problems, we propose a distributed variant of the popular stochastic approximation. Our main contribution is to establish the convergence of the distributed stochastic approximation. In addition, we provide an explicit formula for the rate of convergence of our proposed method as a function of the underlying network topology and the number of quantized bits.  

As mentioned, the distributed stochastic approximation considered in this paper can be viewed as a distributed two-time-scale algorithm. Thus, we believe that the proposed algorithm can be extended for solving other problems, such as, distributed reinforcement learning over multi-agent systems. This an interesting topic which we leave for our future studies.

\bibliographystyle{ACM-Reference-Format}
\bibliography{refs}

%\newpage

%!TEX root = random_quantization.tex

\appendix

%{\Large\bf\centering Supplemental Documents}

\section{Proofs of Results in Section \textbf{4.1} }
In this section, we provide the analysis for the results stated in Section \ref{sec:preliminaries}. To do that, we utilize the result on the properties of the projection studied in \cite{NedicOP2010}, stated as follows.
\begin{lem}[Lemma $1$ \cite{NedicOP2010}]\label{apx_lem:projection}
Let $\Xcal$ be a nonempty closed convex set in $\Rset^d$. Then, we have for any $\xbf\in\mathbb{R}^d$
\begin{enumerate}[leftmargin = 6mm]
\item[(a)] $(\Pcal_{\Xcal}[\xbf]-\xbf)^T(\xbf-\ubf)\leq -\|\Pcal_{\Xcal}[\xbf]-\xbf\|^2\;$ for all $\ubf\in\Xcal$.
\item[(b)] $\|\Pcal_{\Xcal}[\xbf]-\ubf\|^2\leq \|\xbf-\ubf\|^2 -\|\Pcal_{\Xcal}[\xbf]-\xbf\|^2\;$ for all $\ubf\in\Xcal$.
\end{enumerate}
\end{lem}

\subsection{Proof of Lemma \ref{lem:consensus_bound}}
Since $\Ybf(k) = \Xbf(k) -\1\xbarbf^T(k) = \Wbf\Xbf(k)$, Eqs.\ \eqref{anlaysis:Xupdate} and \eqref{analysis:xbar} give
\begin{align}
\Ybf(k+1) &= \Wbf\Xbf(k+1)\notag\\
&= (1-\beta(k))\Ybf(k) + \beta(k)\Abf\Wbf\Qbf(k)\notag\\
&\qquad -\alpha(k)\Wbf\Gbf(k) - \Wbf\Xibf(k).\label{lem_consensus:Eq1}
\end{align}
By the Cauchy-Schwarz inequality we have for any $\eta > 0$ and $a,b\in\Rset$   
\begin{align}
(a+b)^2\leq (1+\eta)a^2 + (1+1/\eta)b^2.\label{lem_consensus:Eq1a}
\end{align}
Thus, by taking the $2$-norm square of Eq.\ \eqref{lem_consensus:Eq1} and using Eq.\ \eqref{lem_consensus:Eq1a} with $\eta = 1 + (1-\sigma_2)\beta(k)$ we have
\begin{align}
&\|\Ybf(k+1)\|^2\notag\\ 
&= \Bigg\|(1-\beta(k))\Ybf(k) + \beta(k)\Abf\Wbf\Qbf(k)\notag\\
&\qquad -\alpha(k)\Wbf\Gbf(k)-\Wbf\Xibf(k)\Bigg\|^2\nonumber\\
&\leq (1+(1-\sigma_2)\beta(k)) \left\|(1-\beta(k))\Ybf(k) + \beta(k)\Abf\Wbf\Qbf(k)\right\|^2\notag\\
&\qquad + \left(1+\frac{1}{(1-\sigma_2)\beta(k)}\right)\left\|\alpha(k)\Wbf\Gbf(k)+\Wbf\Xibf(k)\right\|^2\notag\\
&\leq (1+(1-\sigma_2)\beta(k)) \left\|(1-\beta(k))\Ybf(k) + \beta(k)\Abf\Wbf\Qbf(k)\right\|^2\notag\\
&\qquad + \left(2+\frac{2}{(1-\sigma_2)\beta(k)}\right)\left\|\alpha(k)\Wbf\Gbf(k)\right\|^2\notag\\
&\qquad + \left(2+\frac{2}{(1-\sigma_2)\beta(k)}\right)\left\|\Wbf\Xibf(k)\right\|^2,\label{lem_consensus:Eq2}
\end{align}
where the last inequality we use Eq.\ \eqref{lem_consensus:Eq1a} with $\eta = 1$. First, Proposition \ref{prop:bounded_subg} gives
\begin{align}
\left\|\Wbf\Gbf(k)\right\|^2 \leq \|\Gbf(k)\|^2 \leq L^2.\label{lem_consensus:Eq2a} 
\end{align}
Second, denote by $$\ubf_i(k) = (1-\beta(k))\xbf_i(k) + \beta(k)\sum_{j\in\Ncal_i}a_{ij}\qbf_j(k)\in\Xcal\;\text{a.s.},$$ since $\xbf_i(k),\qbf_i(k)\in\Xcal$, for all $i\in\Vcal$ and $k\geq0$. By Lemma \ref{apx_lem:projection}(b) we have
\begin{align}
&\|\Wbf\Xibf(k)\|^2 \leq \|\Xibf(k)\|^2 = \sum_{i=1}^n\|\xibf_i(k)\|^2\notag\\
&= \sum_{i=1}^n\|\vbf_i(k)-\xbf_i(k+1)\|^2 = \sum_{i=1}^n\|\vbf_i(k)-\Pcal_{\Xcal}[\vbf_i(k)]\|^2\notag\\ 
&\leq \sum_{i=1}^n \|\vbf_i(k)-\ubf_i(k)\|^2= \sum_{i=1}^n\left\|\alpha(k)\gbf_i(k)\right\|^2\notag\\
&\leq \sum_{i=1}^n\alpha^2(k)L_i^2\leq  L^2\alpha^2(k),\label{lem_consensus:Eq2b} 
\end{align}
where the last inequality is due Proposition \ref{prop:bounded_subg}. 
Third, by Eq.\ \eqref{notation:random_quantization} we have
\begin{align}
\|\Ebf(k)\|^2 &= \sum_{i=1}^{n}\|\xbf_i(k)-\qbf_i(k)\|^2\leq \sum_{i=1}^{n}\sum_{\ell=1}^d\Big(\Delta^{\ell}\Big)^2\notag\\
&\leq \sum_{i=1}^{n}\Delta^2 = n\Delta^2, \label{lem_consensus:Eq2c1}
\end{align}
where recall that $\Delta^{\ell}$ is the distance of $\ell$-coordinate of $\xbf_i-\qbf_i$. In addition, by Eq.\ \eqref{const:sigma_2} we have
\begin{align}
&\|(1-\beta(k))\Ybf(k) + \beta(k)\Abf\Ybf(k)\|^2\notag\\
&\qquad \leq \|(1-(1-\sigma_2)\beta(k))\Ybf(k)\|.    \label{lem_consensus:Eq2c2}
\end{align}
Thus, using Eqs.\ \eqref{lem_consensus:Eq2c1} and \eqref{lem_consensus:Eq2c2} we consider 
\begin{align}
&\left\|(1-\beta(k))\Ybf(k) + \beta(k)\Abf\Wbf\Qbf(k)\right\|^2\notag\\
&= \|(1-\beta(k))\Ybf(k) + \beta(k)\Abf\Ybf(k) +  \beta(k)\Abf\Wbf\Ebf(k)\|^2\notag\\
&= \|(1-\beta(k))\Ybf(k) + \beta(k)\Abf\Ybf(k)\|^2 + \|\beta(k)\Abf\Wbf\Ebf(k)\|^2\notag\\
&\quad + 2\beta(k)\Big((1-\beta(k))\Ybf(k) + \beta(k)\Abf\Ybf(k)\Big)^T\Abf\Wbf\Ebf(k)\notag\\
&\stackrel{\eqref{const:sigma_2}}{\leq}  \|(1-\beta(k))\Ybf(k) + \beta(k)\Abf\Ybf(k)\|^2 + \sigma_2^2\beta^2(k)\|\Ebf(k)\|^2\notag\\
&\quad + 2\beta(k)\Big((1-\beta(k))\Ybf(k) + \beta(k)\Abf\Ybf(k)\Big)^T\Abf\Wbf\Ebf(k)\notag\\
&\underset{\eqref{lem_consensus:Eq2c2}}{\stackrel{\eqref{lem_consensus:Eq2c1}}{\leq}} \|(1-(1-\sigma_2)\beta(k))\Ybf(k)\|^2 + n\sigma_2^2\Delta^2\beta^2(k)\notag\\
&\quad + 2\beta(k)\Big((1-\beta(k))\Ybf(k) + \beta(k)\Abf\Ybf(k)\Big)^T\Abf\Wbf\Ebf(k). \label{lem_consensus:Eq2c}
\end{align}
Note that by Eq. \eqref{notation:random_quantization} we have $\Eset[\,\Ebf(k)\,] = 0$. Thus, taking the conditional expectation of Eq.\ \eqref{lem_consensus:Eq2} w.r.t. $\Fcal_k$ and using Eqs.\ \eqref{lem_consensus:Eq2a}, \eqref{lem_consensus:Eq2b}, and  \eqref{lem_consensus:Eq2c} yields Eq.\ \eqref{lem_consensus:UpperBound}, i.e.,
\begin{align*}
&\Eset[\|\Ybf(k+1)\|^2\,|\,\Fcal_k]\notag\\
&\qquad\leq(1+(1-\sigma_2)\beta(k)) (1-(1-\sigma_2)\beta(k))^2\|\Ybf(k)\|^2 \notag\\
&\qquad\qquad + n\sigma_2^2\Delta^2\beta^2(k) + 4L^2\left(1+\frac{1}{(1-\sigma_2)\beta(k)}\right)\alpha^2(k)\notag\\
&\qquad\leq  (1-(1-\sigma_2)\beta(k))\|\Ybf(k)\|^2 \notag\\
&\qquad\qquad + n\sigma_2^2\Delta^2\beta^2(k)+\frac{4L^2(\beta(0)+1)}{(1-\sigma_2)}\frac{\alpha^2(k)}{\beta(k)},
\end{align*}
where the last inequality is because $\beta(k)$ is nonincreasing.

Finally, taking the expectation on both sides of the preceding relation and summing up over $k=0,\ldots,K$ for some $K$ immediately give Eq.\ \eqref{lem_consensus:UpperBound_rate}. 

\subsection{Proof of Lemma \ref{lem:consensus}}
Recall from Eq.\ \eqref{lem_consensus:UpperBound} that 
\begin{align*}
&\Eset[\|\Ybf(k+1)\|^2\,|\,\Fcal_k]\notag\\
&\qquad\leq \|\Ybf(k)\|^2 -(1-\sigma_2)\beta(k)\|\Ybf(k)\|^2 \notag\\
&\qquad\qquad + n\sigma_2^2\Delta^2\beta^2(k)+\frac{8L^2}{(1-\sigma_2)}\frac{\alpha^2(k)}{\beta(k)}.
\end{align*}
In addition, since $\alpha(k)$ and $\beta(k)$ satisfy Eq.\ \eqref{lem_consensus:stepsizes} we can apply Lemma \ref{lem:Martingale} to the preceding equation. Thus, we obtain $\{\Ybf(k)\}$ converges a.s. and $\sum_{k=0}^{\infty}\beta(k)\|\Ybf(k)\|^2 < \infty$. Since $\sum_{t=0}^{\infty}\beta(k) = \infty$, the preceding relation immediately gives 
\begin{align*}
\lim_{k\rightarrow\infty}\|\xbf_i(k)-\xbarbf(k)\| =0 \qquad\text{a.s.,}\quad \text{ for all } i\in\Vcal.      
\end{align*}
\subsection{Proof of Lemma \ref{lem:opt_dist}}
First, Eq.\ \eqref{analysis:xbar} gives
\begin{align*}
&r(k+1) = \|\vbarbf(k)-\xibarbf(k)-\xbf^*\|^2\notag\\ 
&= \left\|\begin{array}{ll} (1-\beta(k))\xbarbf(k)-\xbf^* -\frac{\alpha(k)}{n}\sum_{i=1}^n \gbf_i(\xbf_i(k))\\ 
\qquad\qquad + \beta(k)\qbarbf(k)-\xibarbf(k)
\end{array}
\right\|^2\nonumber\\
&= \left\|\xbarbf(k)-\xbf^*-\frac{\alpha(k)}{n}\sum_{i=1}^n \gbf_i(x_i(k))\right\|^2\notag\\ 
&\quad + \left\|\beta(k)\ebarbf(k)\right\|^2 + \left\|\xibarbf(k)\right\|^2 - 2\beta(k)\ebarbf^T(k)\xibarbf(k) \notag\\
&\quad + 2\beta(k)\ebarbf^T(k)\left(\xbarbf(k)-\xbf^* -\frac{\alpha(k)}{n}\sum_{i=1}^n \gbf_i(\xbf_i(k))\right)\notag\\
&\quad - 2\xibarbf^T(k)\left(\xbarbf(k)-\xbf^* -\frac{\alpha(k)}{n}\sum_{i=1}^n \gbf_i(\xbf_i(k))\right),
\end{align*}
which by taking the conditional expectation w.r.t $\Fcal_k$ yields
\begin{align}
&\Eset\left[\,r(k+1)\,|\Fcal_k\,\right]\notag\\
&= \left\|\xbarbf(k)-\xbf^*-\frac{\alpha(k)}{n}\sum_{i=1}^n \gbf_i(x_i(k))\right\|^2\notag\\ 
&\quad + \left\|\beta(k)\ebarbf(k)\right\|^2 + \left\|\xibarbf(k)\right\|^2 \notag\\
&\qquad - 2\xibarbf^T(k)\Big(\xbarbf(k)-\xbf^*\Big)\notag\\ 
&\qquad + \frac{2\alpha(k)}{n}\xibarbf^T(k) \sum_{i=1}^n \gbf_i(\xbf_i(k)).\label{lem_dist_opt:Eq1}
\end{align}
First, we use Eq.\ \eqref{lem_consensus:Eq2b} to have $\|\xibf_i(k)\|\leq L_i\alpha(k)$. Thus, by Eq. \eqref{notation:random_quantization} we obtain
\begin{align}
& \left\|\beta(k)\ebarbf(k)\right\|^2 + \left\|\xibarbf(k)\right\|^2\notag\\   
&\quad = \beta^2(k)\left\|\frac{1}{n}\sum_{i=1}^n\ebf_i(k) \right\|^2 + \left\|\frac{1}{n}\sum_{i=1}^n\xibf_i(k) \right\|^2\notag\\
&\quad \leq \frac{\beta^2(k)}{n}\sum_{i=1}^n\left\|\ebf_i\right\|^2 + \frac{\alpha^2(k)}{n}\sum_{i=1}^nL_i^2 \leq \Delta^2\beta^2(k) + \frac{L^2\alpha^2(k)}{n}\cdot  \label{lem_dist_opt:Eq1a}
\end{align}
Second, using Proposition \ref{prop:bounded_subg} gives
\begin{align}
&\frac{2\alpha(k)}{n}\xibarbf^T(k) \sum_{i=1}^n \gbf_i(\xbf_i(k))\notag\\
&\quad \leq \frac{2\alpha(k)}{n}\frac{1}{n}\sum_{i=1}^n L_i\alpha(k)\sum_{i=1}^nL_i = \frac{2L^2\alpha^2(k)}{n^2}\cdot   \label{lem_dist_opt:Eq1b}
\end{align}
Third, for convenience let $\rbf_i(k),\hbf_i(k)$ be defined as 
\begin{align*}
\rbf_i(k) &= \frac{\beta(k)}{1-\beta(k)}\vbf_i(k)- \frac{\beta^2(k)}{1-\beta(k)}\sum_{j=1}^na_{ij}\qbf_j(k)\notag\\
%&= \beta(k)\xbf_i(k) -\frac{\alpha(k)\beta(k)}{1-\beta(k)}\gbf_i(\xbf_i(k))\\
\hbf_i(k) &= (1-\beta(k))\xbf^* + \beta(k)\sum_{j=1}^na_{ij}\qbf_j(k)\; \in\Xcal\quad \text{a.s.}       
\end{align*}
Using $\rbf_i(k),\hbf_i(k)$ and by Lemma \ref{apx_lem:projection}(a) we consider 
\begin{align}
&= -\frac{2}{n\beta(k)}\sum_{i=1}^n \xibf_i^T(k)\Big(\rbf_i(k)-\beta(k)\xbf^*\Big)\notag\\
&= -\frac{2}{n(1-\beta(k))}\sum_{i=1}^n \xibf_i^T(k)\Big(\vbf_i(k)-\hbf_i(k) \Big)\notag\\
&= \frac{2}{n(1-\beta(k))}\sum_{i=1}^n \Big(\Pcal_{\Xcal}[\vbf_i(k)]-\vbf_i(k)\Big)\Big(\vbf_i(k)-\hbf_i(k) \Big)\notag\\
&\leq -\frac{2}{n(1-\beta(k))}\sum_{i=1}^{n}\|\xibf_i(k)\|^2\leq 0.
\label{lem_dist_opt:Eq1c1}
\end{align}
Thus, we have
\begin{align}
&- 2\xibarbf^T(k)\Big(\xbarbf(k)-\xbf^*\Big)\notag\\ 
&= -\frac{2}{n\beta(k)}\sum_{i=1}^n \xibf_i^T(k)\Big(\beta(k)\xbarbf(k)-\beta(k)\xbf^*\Big)\notag\\
& = -\frac{2}{n\beta(k)}\sum_{i=1}^n \xibf_i^T(k)\Big(\beta(k)\xbarbf(k) - \rbf_i(k)\Big)\notag\\ 
&\qquad  -\frac{2}{n\beta(k)}\sum_{i=1}^n \xibf_i^T(k)\Big(\rbf_i(k)-\beta(k)\xbf^*\Big)\notag\\
&\stackrel{\eqref{lem_dist_opt:Eq1c1}}{\leq} -\frac{2}{n\beta(k)}\sum_{i=1}^n \xibf_i^T(k)\beta(k)(\xbarbf(k)-\xbf_i(k))\notag\\ 
&\qquad + \frac{2}{n\beta(k)}\sum_{i=1}^n \xibf_i^T(k)\frac{\alpha(k)\beta(k)}{1-\beta(k)}\gbf_i(\xbf_i(k))\allowdisplaybreaks\notag\\ 
&\leq \frac{2}{n}\sum_{i=1}^n\| \xibf_i^T(k)\|\,\|\ybf_i(k)\| + \frac{2\alpha(k)}{n(1-\beta(k))}\sum_{i=1}^n L_i\|\xibf_i(k)\|\notag\\ 
&\stackrel{\eqref{lem_consensus:Eq2b}}{\leq} \frac{2}{n}\sum_{i=1}^nL_i\alpha(k)\,\|\ybf_i(k)\| + \frac{2L^2\alpha^2(k)}{n(1-\beta(k))}\notag\\
&\leq \frac{L^2\alpha^2(k)}{n\beta(k)} + \frac{\beta(k)\|\Ybf(k)\|^2}{n} + \frac{2L^2\alpha^2(k)}{n(1-\beta(0))}, 
\label{lem_dist_opt:Eq1c}
\end{align}
where the las inequality is due to the Cauchy-Schwarz inequality and $1-\beta(k)\geq 1-\beta(0)$. Next, consider the following  
\begin{align}
&\left\|\xbarbf(k)-\xbf^*-\frac{\alpha(k)}{n}\sum_{i=1}^n \gbf_i(x_i(k))\right\|^2\notag\\
&= \left\|\xbarbf(k)-\xbf^*\right\|^2 + \left\|\frac{\alpha(k)}{n}\sum_{i=1}^n \gbf_i(x_i(k))\right\|^2\notag\\
&\qquad - \frac{2\alpha(k)}{n}\sum_{i=1}^n \gbf_i^T(x_i(k))\Big(\xbarbf(k)-\xbf^*\Big)\allowdisplaybreaks\notag\\
&\leq \left\|\xbarbf(k)-\xbf^*\right\|^2 + \frac{L^2\alpha^2(k)}{n^2}\notag\\
&\qquad - \frac{2\alpha(k)}{n}\sum_{i=1}^n \gbf_i^T(x_i(k))\Big(\xbarbf(k)-\xbf_i(k)+\xbf_i(k)-\xbf^*\Big)\notag\\
&\leq r(k) + \frac{L^2\alpha^2(k)}{n^2}+\frac{L^2\alpha^2(k)}{n\beta(k)} + \frac{\beta(k)\|\Ybf(k)\|^2}{n}\notag\\
&\qquad - \frac{2\alpha(k)}{n}\sum_{i=1}^n \gbf_i^T(x_i(k))\Big(\xbf_i(k)-\xbf^*\Big),
\label{lem_dist_opt:Eq1d}
\end{align}
where the last inequality is due to
\begin{align*}
&- \frac{2\alpha(k)}{n}\sum_{i=1}^n \gbf_i^T(x_i(k))\Big(\xbarbf(k)-\xbf_i(k)\Big)\notag\\ 
&\leq \sum_{i=1}^n \frac{\alpha^2(k)}{n\beta(k)}\|\gbf_i(\xbf_i(k))\|^2 +  \sum_{i=1}^n \beta(k)\|\xbarbf(k)-\xbf_i(k)\|^2\notag\\
&\leq \frac{L^2\alpha^2(k)}{n\beta(k)} + \frac{\beta(k)\|\Ybf(k)\|^2}{n}.
\end{align*}
% By the convexity of $f_i$ we have
% \begin{align*}
% &- \frac{2\alpha(k)}{n}\sum_{i=1}^n \gbf_i^T(x_i(k))\Big(\xbf_i(k)-\xbf^*\Big) \notag\\
% & \leq \frac{-2\alpha(k)}{n}\sum_{i=1}^n f_i(\xbf_i(k)) - f_i(\xbf^*)\notag\\
% &= \frac{-2\alpha(k)}{n}\sum_{i=1}^nf_i(\xbf_i(k)) -f_i(\xbarbf(k))+f_i(\xbarbf(k)) - f_i(\xbf^*)\notag\\
% &\leq \frac{2}{n}\sum_{i=1}^n\alpha(k)L_i\,\|\ybf_i(k)\,\|-\frac{2\alpha(k)}{n}\sum_{i=1}^nf_i(\xbarbf(k)) - f_i(\xbf^*)\notag\\
% &\leq \frac{L^2\alpha^2(k)}{n\beta(k)} + \frac{\beta(k)\|\Ybf(k)\|^2}{n} -\frac{2\alpha(k)(f(\xbar(k))  - f^*)}{n},
% \end{align*}
% which when substituting into Eq.\ \eqref{lem_dist_opt:Eq1d1} gives 
% \begin{align}
% &\left\|\xbarbf(k)-\xbf^*-\frac{\alpha(k)}{n}\sum_{i=1}^n \gbf_i(x_i(k))\right\|^2\notag\\
% &\leq r(k) + \frac{L^2\alpha^2(k)}{n^2}+\frac{2L^2\alpha^2(k)}{n\beta(k)} + \frac{2\beta(k)\|\Ybf(k)\|^2}{n}\notag\\
% &\qquad -\frac{2\alpha(k)}{n}(f(\xbar(k))  - f^*).
% \label{lem_dist_opt:Eq1d}
% \end{align}
We now substitute Eqs.\ \eqref{lem_dist_opt:Eq1a}, \eqref{lem_dist_opt:Eq1b},  \eqref{lem_dist_opt:Eq1c}, and \eqref{lem_dist_opt:Eq1d} into Eq.\ \eqref{lem_dist_opt:Eq1} to have Eq.\ \eqref{lem_opt_dist:Ineq}, i.e., 
\begin{align*}
&\Eset\left[\,r(k+1)\,|\,\Fcal_k\right]\notag\\ 
&\leq r(k) + \frac{L^2\alpha^2(k)}{n^2}+\frac{L^2\alpha^2(k)}{n\beta(k)} + \frac{\beta(k)\|\Ybf(k)\|^2}{n}\notag\\
&\qquad + \Delta^2\beta^2(k) + \frac{L^2\alpha^2(k)}{n} + \frac{2L^2\alpha^2(k)}{n^2}\notag\\
&\qquad + \frac{L^2\alpha^2(k)}{n\beta(k)} + \frac{\beta(k)\|\Ybf(k)\|^2}{n} + \frac{2L^2\alpha^2(k)}{n(1-\beta(0))}\notag\\
&\qquad - \frac{2\alpha(k)}{n}\sum_{i=1}^n \gbf_i^T(x_i(k))\Big(\xbf_i(k)-\xbf^*\Big)\allowdisplaybreaks\notag\\
&\leq r(k) +\frac{2L^2\alpha^2(k)}{n\beta(k)} + \Delta^2\beta^2(k)\notag\\
&\qquad +\frac{6L^2\alpha^2(k)}{n(1-\beta(0))} + \frac{2\beta(k)\|\Ybf(k)\|^2}{n}\notag\\
&\qquad - \frac{2\alpha(k)}{n}\sum_{i=1}^n \gbf_i^T(x_i(k))\Big(\xbf_i(k)-\xbf^*\Big).
 \end{align*}

\end{document}